\documentclass[a4paper,10pt]{article}
\usepackage{amsmath,amssymb,amsthm,dsfont}
\usepackage[english]{babel}

\newtheorem{theorem}{Theorem}[section]
\newtheorem{lemma}[theorem]{Lemma}
\newtheorem{proposition}[theorem]{Proposition}
\newtheorem{corollary}[theorem]{Corollary}
\theoremstyle{definition}
\newtheorem{definition}[theorem]{Definition}
\newtheorem{remark}[theorem]{Remark}
\newtheorem{example}[theorem]{Example}
\newtheorem{assumption}[theorem]{Assumption}

\newcommand{\R}{\mathbb R}                   

\newcommand{\C}{\mathbb C}                   

\newcommand{\scal}[1]{\langle #1 \rangle}

\DeclareMathOperator*{\slim}{s-lim}

\begin{document}

\title{\sc Uniform resolvent and Strichartz estimates for Schr\"odinger equations with critical singularities}

\author{Jean-Marc Bouclet    \and Haruya Mizutani  }

\AtEndDocument{\bigskip{\footnotesize
\textsc{Institut de Math\'ematiques de Toulouse (UMR CNRS 5219),  Universit\'e Paul Sabatier, 118 route de Narbonne, F-31062 Toulouse FRANCE} \par\textit{E-mail address}: \texttt{jean-marc.bouclet@math.univ-toulouse.fr} \par}}
\AtEndDocument{\bigskip{\footnotesize
\textsc{Department of Mathematics, Graduate School of Science, Osaka University, Toyonaka, Osaka 560-0043, Japan} \par\textit{E-mail address}: \texttt{haruya@math.sci.osaka-u.ac.jp} \par}}


\date{ }

\maketitle

\begin{abstract}
This paper deals with global dispersive properties of Schr\"odinger equations with real-valued potentials  exhibiting critical singularities, where our class of potentials is more general than inverse-square type potentials and includes several anisotropic potentials. We first prove weighted resolvent estimates, which are uniform with respect to the energy, with a large class of weight functions in Morrey-Campanato spaces. Uniform Sobolev inequalities in Lorentz spaces are also studied. The proof employs the iterated resolvent identity and a classical multiplier technique. As an application, the full set of global-in-time Strichartz estimates including the endpoint case is derived. In the proof of Strichartz estimates, we develop a general criterion on perturbations ensuring that both homogeneous and inhomogeneous endpoint estimates can be recovered form  resolvent estimates. Finally, we also investigate uniform resolvent estimates for long range repulsive potentials with critical singularities by using an elementary version of the Mourre theory.

\end{abstract}



\section{Introduction}
\setcounter{equation}{0}
Given a self-adjoint operator $H$ on a Hilbert space $\mathcal H$ and $z\in \rho(H)$, the resolvent $(H-z)^{-1}$ is  a bounded operator on $\mathcal H$ and satisfies $$||(H-z)^{-1}||_{\mathcal H\to \mathcal H}=\frac{1}{\mathrm{dist}(z,\sigma(H))}$$ 
by the spectral theorem. Hence there is no hope to obtain the estimate in the operator norm sense which is uniform with respect to $z$ close to the spectrum of $H$. However, uniform estimates in $z$ can be recovered for many important operators by considering, {\it e.g.}, the weighted resolvent $w(H-z)^{-1}w^*$ with an appropriate closed operator $w$. Such {\it uniform resolvent estimates}  play a fundamental role in the study of broad areas including spectral and scattering theory for Schr\"odinger equations. In particular, as observed by Kato \cite{Kato} and Rodnianski-Schlag \cite{RoSc}, uniform resolvent estimates are closely connected to global-in-time dispersive estimates such as time-decay estimates or Strichartz estimates which are important tools in the scattering theory for nonlinear dispersive partial differential equations, see monographs \cite{Caz,Tao}. 

In this paper we study uniform resolvent estimates and their applications to global-in-time Strichartz estimates for Schr\"odinger operators $$H=-\Delta+V(x)$$ on $L^2(\R^n)$ with real-valued potentials  $V(x)$ exhibiting critical singularities, where $\Delta$ is the usual Laplacian. Typical examples of critical potentials we have in mind are inverse-square type potentials, {\it i.e.}, $ |x|^2 V \in L^{\infty}$, which represent a borderline case for the validity of these estimates (see \cite{Duy1,GVV}). Note however that our class of potentials includes several examples so that $|x|^2V\notin L^\infty$. 

If $V$ decays sufficiently fast at infinity and has enough regularity, say $V$ has a finite global Kato norm (see \cite{RoSc}), then there is a vast literature on both uniform resolvent estimates with various type of weights $w$ and their applications to global-in-time Strichartz estimates under certain regularity conditions on the zero energy, see \cite{JeKa,JeNe,GoSc1,RoTa} for resolvent estimates and  \cite{RoSc,BaDu,DaFa,EGS1,EGS2,Gol,DaFa,MMT,DFVV,Bec} for Strichartz estimates, and references therein. On the other hand, when $V$ has at least one critical singularity and decays like $|x|^{-2}$ at infinity, although there are still many results on resolvent estimates (see \cite{PeVe,BPST2,Fan,Mochizuki,BVZ} and references therein), the choice of $w$ has been limited to a specific type of weights 
which restricts the range of applications. In particular, in contrast to the case of inverse-square type potentials (for which we refer to \cite{PST1,PST2,BPST1,BPST2,KoTr,FFFP} and references therein), there seems to be no previous literature on global-in-time Strichartz estimates for large potentials with critical singularities which are not of inverse-square type (see a recent result \cite{KiKo} for small potentials with critical singularities). Finally, if $V$ has at least one critical singularity and decays slower than $|x|^{-2}$ at infinity, there seems to be no positive results on both uniform resolvent and global-in-time dispersive estimates, while there are several positive results on resolvent estimates if $V$ is less singular  (see \cite{Nak,FoSk}). 

In the light of those observations, the purpose of this paper is twofold. 

The first purpose is to investigate uniform estimates for the weighted resolvent $w(H-z)^{-1}w$ with potentials $V$ exhibiting critical singularities and with a wide class of weight functions $w$ in Morrey-Campanato spaces. We also consider uniform estimates for $(H-z)^{-1}$ in $L^p$ spaces (or more generally, Lorentz spaces),  known as uniform Sobolev inequalities which are due to \cite{KRS} for constant coefficient operators. Our admissible class of potentials includes several anisotropic potentials, which are more general than inverse-square type potentials, so that $V$ can have a critical singularity of type $|x|^{-2}$ at the origin and multiple Coulomb type singularities away from the origin. 

As an application, we show the full set of global-in-time Strichartz estimates (including both homogeneous and inhomogeneous endpoint cases) for the above class of potentials, which improves upon the previous references \cite{BPST1,BPST2,KiKo} in the following directions. On one hand, we can consider a larger class of admissible potentials with critical singularities. More importantly, we provide a general criterion on potentials  ensuring that both {\it homogeneous and inhomogeneous endpoint} Strichartz estimates can be recovered from uniform resolvent estimates. More precisely we develop an abstract smooth perturbation method which enables us to deduce the full set of Strichartz estimates for the perturbed operator $H$ from corresponding estimates for the unperturbed operator $H_0$ and the uniform Sobolev inequality for the resolvent $(H-z)^{-1}$. This extends the previous techniques by \cite{RoSc,BPST2,BaDu} to a quite general setting. 

Another important problem is to investigate the validity of global-in-time  Strichartz estimates for Schr\"odinger operators with long range potentials with singularities (e.g. in the Coulombic case) in view of their applications to the study of long-time behaviors of the Hartree equation with external potentials, which is a nonlinear model for the quantum dynamics of an atom. As a step toward this problem, the second purpose of the paper is to consider resolvent estimates for long range repulsive potentials with critical singularities. More specifically, we show how some elementary version of the Mourre theory can be used to obtain uniform resolvent estimates in this strongly singular case (the potentials and weight functions in \cite{Nak,FoSk} were not as singular as ours).

Finally, we mention several possible applications of the results in this paper. As already observed, our Strichartz estimates could be used to study scattering theory for nonlinear Schr\"odinger equations with singular potentials. For recent results in this context, we refer to \cite{ZhZh,KMVZZ,KMVZ} in which the case with the inverse-square potential was studied. Another range of applications, which will be considered in a subsequent work \cite{Miz1}, is about eigenvalues estimates for Schr\"odinger operators with complex-valued potentials. As already observed by \cite{Fra}, uniform resolvent estimates with singular weights are an important input in the derivation of eigenvalues bounds with singular potentials.

\section{Notation and main results} \label{sectionresults}
\setcounter{equation}{0}

Let us introduce the class of potentials we will use. We distinguish the dimension $n=2$ from the case $ n \geq 3 $. For  $ 1 \leq \sigma \leq q < \infty$, we consider the Morrey-Campanato norms
$$ || W ||_{M^{q,\sigma}} := \sup_{x \in {\mathbb R}^n \atop r > 0}  r^{\frac{n}{q}} \Big( r^{-n} \int_{|y-x|<r} |W(y)|^{\sigma} dy \Big)^{\frac{1}{\sigma}} . $$
The space $ M^{q,\sigma} $ is the set of measurable functions with finite $ || \cdot ||_{M^{q,\sigma}} $ norm.
For $ 1 \leq q , \sigma \leq \infty $, we will use the Lorentz norms
$$ || W ||_{L^{q,\sigma}} =  \big| \big| s^{\frac{1}{q} - \frac{1}{\sigma}} W^*  \big| \big|_{L^{\sigma}  ((0,\infty),ds) }$$
where $ W^* (s) $ is the decreasing rearrangement of $W$ (see paragraph \ref{sectionnotationLorentz} below for basic properties of Morrey-Campanato and Lorentz spaces).  
We simply recall here that these norms have the same scaling as the usual $ L^q $ norm, namely they are invariant under the scaling $ W (x) \mapsto \lambda^{\frac{n}{q}} W (\lambda x) $. Also note that $L^{q,\infty}\subset M^{q,\sigma}$ if $1\le q<\infty$ and $1\le \sigma<q$. Let us set
\begin{align*}
{\mathcal X}_n^{\sigma} &:= \big\{ V:\R^n\to\R \ | \ |x| V \in M^{n,2\sigma} \ \mbox{and} \ x \cdot \nabla V \in M^{\frac{n}{2},\sigma} \big\}\quad\text{if $  n \geq 3 $ and $  {(n-1)}/{2} < \sigma \leq {n}/{2} $},\\
{\mathcal X}_2 &:= \big\{ V:\R^2\to\R \ | \ |x|^2 (x \cdot \nabla)^{\ell} V \in L^{\infty}({\mathbb R}^2), \ \ \ell = 0 , 1,2 \big\}\quad\text{if $n=2$}. 
\end{align*}
\begin{assumption} ($ n \geq 3 $)
\label{assumption_1}
 There exists $\delta_0>0$ such that for all $f\in C_0^ {\infty}(\R^n \setminus 0) $,
\begin{align}
\label{assumption_1_1}
 \langle (-\Delta+V)f,f \rangle &\ge \delta_0 ||\nabla f||_{L^2}^2,\\
\label{assumption_1_2}
\langle (-\Delta-V-x\cdot (\nabla V))f,f \rangle &\ge \delta_0||\nabla f ||_{L^2}^2. 
\end{align}
Here and below, $ \langle f,g \rangle = \int f (x) \overline{g(x)}dx $ is the usual $ L^2 $ inner product.
\end{assumption}

\begin{example} ($n\ge3$) 
\label{example}
A typical example satisfying Assumption \ref{assumption_1} is the inverse-square potential $-c_0|x|^{-2}$ with $c_0<(n-2)^2/4$. Our class also includes inverse-square type potentials $V$ such that 
$$
|x|^2V\in L^\infty,\quad |x|^2x\cdot\nabla V\in L^\infty,\quad 
V\ge -c_0|x|^{-2},\quad
-V-x\cdot\nabla V\ge -c_0|x|^{-2}.
$$ 
In these cases \eqref{assumption_1_1} and \eqref{assumption_1_2} follow from classical Hardy's inequality:
$$
\frac{(n-2)^2}{4}\big|\big||x|^{-1}f\big|\big|_{L^2}^2\le ||\nabla f||_{L^2}^2,\quad f\in C_0^\infty(\R^n\setminus0).
$$
Moreover,  we have $V\in \mathcal X^\sigma_n\cap L^{\frac n2,\infty}$ since $|x|^{-1}\in L^{n,\infty} \subset M^{n,2\sigma}$ for all $1\le \sigma<n/2$. 

Assumption \ref{assumption_1} is actually more general enough to accommodate several anisotropic potentials so that $|x|^2V\notin L^\infty$. For instance, we let $c_1,c_2>0$, $\alpha\in\R^n$ and $\chi\in C^1(\R)$ such that $0\le \chi\le 1$ and $|\chi^{(k)}(t)|\le |t|^{-k-1}$ for $|t|\ge1$. Define
$$
V(x)=\Big(-\frac{(n-2)^2}{4}+c_1\Big)|x|^{-2}-c_2\chi(|x-\alpha|)|x-\alpha|^{-1}. 
$$
Then $V\in \mathcal X_n^\sigma\cap L^{\frac n2,\infty}$ and $V$ satisfies Assumption \ref{assumption_1} with $\delta_0=c_1-c_2(2+\sup|\chi'|)(|\alpha|+1)$ if 
$
0<c_2<{c_1}{(2+\sup|\chi'|)^{-1}(|\alpha|+1)^{-1}}
$. One can also consider multiple Coulomb type singularities. 
\end{example}

\begin{assumption} ($n = 2$) There exists $ \delta_0 > 0 $ such that, almost everywhere on  $ {\mathbb R}^2  $,
\label{assumption_2}
\begin{equation}
\begin{aligned}
 \label{positifdim2}
 V>0, \quad
 \delta_0^{-1 }V  \geq -x \cdot \nabla V \geq (1+\delta_0)V. 
\end{aligned}
\end{equation}
Furthermore, $V^{-1}$ is locally integrable in $\R^2$. 
\end{assumption}

\begin{example} ($n=2$) 
A typical example of $V\in \mathcal X_2$ satisfying Assumption \ref{assumption_2} is $V=V_1+V_2$ such that
\begin{align*}
V_1(x)=a(\theta)r^{-\nu}\langle r\rangle ^{\nu-\mu},\quad 
V_2(x)=a(\theta)r^{-2}(1+(\log r)^2)^{-\delta},\quad r=|x|,\ \theta=x/r,
\end{align*}
where $a\in L^\infty(\mathbb S^1)$ such that $a>c_0$ on $\mathbb S^1$ with some $c_0>0$, $\mu\ge 2$, $\nu\in(1,2]$ and $\delta\ge0$.  Indeed, 
$$
-x\cdot \nabla V_1=\Big(\mu-\frac{\mu-\nu}{1+r^2}\Big)V_1\ge \nu V_1,\quad
-x\cdot \nabla V_2=\Big(2+\delta-\frac{\delta}{1+(\log r)^2}\Big)V_2\ge(2+\delta) V_2.
$$
Furthermore, if $\mu>2,\nu\in(1,2)$ and $\delta>1/2$ then $V_1,V_2\in L^1(\R^2)$. \end{example}

Let us note that both $\mathcal X^\sigma_n$ and $\mathcal X_2$ are invariant by the scaling 
\begin{align}
V (x) \mapsto \lambda^{-2} V (x/\lambda),\quad \label{scalingofpotentials} \lambda>0,
\end{align}
 in the sense that all of norms $ \big| \big| |x| V \big| \big|_{M^{n,2\sigma}} $, $ || x \cdot \nabla V ||_{M^{\frac{n}{2},\sigma}} $ and $ \big| \big|  |x|^2 (x \cdot \nabla)^{\ell} V  \big| \big|_{L^{\infty}} $ are invariant under \eqref{scalingofpotentials}.
Both Assumptions \ref{assumption_1} and \ref{assumption_2} are also invariant under the scaling \eqref{scalingofpotentials}. More precisely, if one of them is satisfied by some $V$, it is still satisfied by $ \lambda^{-2}V(x/\lambda) $ with the same constant $ \delta_0 $. According to this  invariance, all estimates in theorems and corollaries in this section (except Theorem \ref{theorem_3} and Corollary \ref{Corollaire-Mourre}) are invariant under the scaling \eqref{scalingofpotentials}. 

In the sequel, we let $ H$ be the self-adjoint realization of $ - \Delta + V $ defined  in paragraph \ref{realizations}. 
The first result is on uniform weighted resolvent estimates in $L^2$: 

\begin{theorem}[Uniform weighted resolvent estimates] $ $	
\label{theorem_0}

\smallskip

\noindent {\rm (1)} Suppose $n\ge3$ and   $\frac{n-1}{2}<\sigma \leq \frac n2$. Let $V \in {\mathcal X}_n^{\sigma}$ satisfy Assumption \ref{assumption_1}. Then, for any $w_1,w_2\in M^{n,2\sigma}(\R^n)$, $z \in \C \setminus {[0,\infty)}$ and $f \in C_0^{\infty}({\mathbb R}^n)$
\begin{align}
\label{theorem_0_1}
\big| \big| w_1 (H-z)^{-1} w_2   f \big| \big|_{L^2(\R^n)} \leq C||w_1||_{M^{n,2\sigma}}||w_2||_{M^{n,2\sigma}} || f ||_{L^2(\R^n)} 
\end{align}
with some constant $C>0$ independent of $w_1,w_2,f$ and $z$. 
\smallskip

\noindent {\rm (2)} Suppose $n=2$ and $V \in {\mathcal X}_2$ satisfies Assumption \ref{assumption_2}. Then 
 $$  || V^{\frac{1}{2}} (H-z)^ {-1} V^{\frac{1}{2}} f ||_{L^2(\R^2)}  \leq C || f ||_{L^2(\R^2)}, \quad z \in \C \setminus {[0,\infty)}, \ f \in C_0^{\infty}({\mathbb R}^2 \setminus 0)    . $$
\end{theorem}
This theorem means that we have uniform estimates for $ w (H-z)^ {-1} w  $ (for $ n \geq 3 $) and $ V^{\frac{1}{2}} (H-z)^ {-1} V^{\frac{1}{2}}  $ (for $n=2$). To be completely rigorous, the uniform estimates hold for the closure of those weighted resolvents to $ L^2 $; indeed, in general the multiplication by $ w $ or $ V^{\frac{1}{2}} $ are not bounded on $ L^2 $ so the weighted resolvents can not be interpreted (for any fixed $z$) as compositions of bounded operators on $ L^2 $. For completeness, we record here that, for $n=3$, $ w\in L^2_{\rm loc} $ and hence $ w f \in L^2 $ whenever $ f\in C_0^{\infty}(\R^n) $. When $n=2$, $ V^{\frac{1}{2}} \in L^{\infty}_{\rm loc} ({\mathbb R}^2 \setminus 0) $ so  $ V^{\frac{1}{2}} f\in L^2 $ when $ f \in C_0^{\infty}({\mathbb R}^2 \setminus 0)$.

\begin{remark}
\label{remark_theorem_0}
When $w_1=w_2=|x|^{-1}$, \eqref{theorem_0_1} holds for more general potentials. We refer to  Theorem \ref{theorem_resolvent_2main} in Section \ref{section_proof} which improves upon the previous results by \cite{BPST2} (we do not assume $|x|^2V\in L^\infty$) and \cite{BVZ} (see Remark \ref{remark_C} in Appendix \ref{appendix_C}). Compared with this result, the interest of Theorem \ref{theorem_0} (1) is that our class of admissible weights is quite general and particularly includes the weight $w_1=w_2=|V|^{1/2}$. This fact is crucial to apply \eqref{theorem_0_1} to obtain estimates in $L^p$ spaces such as uniform Sobolev and Strichartz estimates (see below) with potentials involving multiple singularities as in Example \ref{example}. 

It is also worth noting that, in contrast to higher dimensional cases $n\ge3$, the two-dimensional free resolvent $(-\Delta_{\R^2}-z)^{-1}$ has a logarithmic singularity at $z=0$ (see, {\it e.g.}, \cite{JeNe}) and hence one cannot hope to obtain uniform estimates in $z$ with any kind of {\it physical} weight $w(x)$. Theorem \ref{theorem_0} (2) thus demonstrates a ``repulsive" effect of the potential $V$ satisfying Assumption \ref{assumption_2}. 
\end{remark}

Let $e^{-itH}$ be the unitary group generated by $H$. For $ F \in L^1_{\rm loc} ({\mathbb R};L^2) $, we define
$$ \Gamma_H F (t) = \int_0^t e^{-i(t-s)H} F (s) ds , $$
and call $ \Gamma_H $ the Duhamel operator associated to $H$. It is defined by means of the Bochner integral.
  Then, for $ \psi \in L^2 $ and $ F \in L^1_{\rm loc} ({\mathbb R};L^2) $, the unique (mild) solution $u(t)$ to the Schr\"odinger equation 
\begin{align}
\label{equation_1}
i\partial_tu=Hu+F(t);\quad u_{|_{t=0}}=\psi,
\end{align}
 is given by the Duhamel formula (see, \emph{e.g.}, \cite[Section 3]{Arendt})
\begin{align}
\label{equation_2}
u(t)=e^{-itH}\psi-i \Gamma_H F (t) . 
\end{align} 
Then Theorem \ref{theorem_0} implies the following result. As usual, when $ {\mathcal B} $ is a Banach space and $ p \geq 1 $, the norm $ || v ||_{L^p ({\mathbb R};{\mathcal B})} $ stands for the $ L^p ({\mathbb R}) $ norm of $ t \mapsto ||v(t) ||_{\mathcal B} $.

\begin{corollary}[$L^2$ space-time estimates] 
\label{corollary_theorem_0}   Under the conditions of Theorem \ref{theorem_0}, the solutions to \eqref{equation_1} satisfy the following estimates.

\smallskip

\noindent {\rm (1)} If $ n \geq 3 $, $ w\in M^{n,2\sigma}(\R^n) $ and $w^{-1}\in L^2_{\mathrm{loc}}(\R^n)$, then there exists $ C > 0 $ such that
$$  || w u ||_{L^2 ({\mathbb R}; L^2(\R^n) )} \leq C || \psi ||_{L^2} + C || w^{-1} F ||_{L^2 ({\mathbb R} ; L^2(\R^n))} $$
for all $ \psi \in L^2(\R^n) $ and $ F \in L^1_{\rm loc} ({\mathbb R};L^2(\R^n)) $ such that $ w^{-1} F \in L^2 ( {\mathbb R} ; L^2(\R^n) ) $.

\smallskip

\noindent {\rm (2)} If $ n \geq 2 $, there exists $ C > 0 $ such that
$$  || V^{\frac{1}{2}} u ||_{L^2 ({\mathbb R}; L^2(\R^2) )} \leq C || \psi ||_{L^2} + C || V^{-\frac{1}{2}} F ||_{L^2 ({\mathbb R} ; L^2(\R^2))} $$
for all  $ \psi \in L^2(\R^2) $ and $ F \in L^1_{\rm loc} ({\mathbb R};L^2(\R^2)) $ such that $ V^{-\frac{1}{2}} F \in L^2 ( {\mathbb R} ; L^2(\R^2) ) $.

\end{corollary}

Under the same assumptions on $V$, we next consider estimates in Lebesgue or Lorentz spaces.

\begin{theorem}[Uniform Sobolev estimates] \label{theoremevariante} Let $ n \geq 3 $ and $ \frac{n-1}{2} < \sigma \leq \frac{n}{2} $. If $ V \in {\mathcal X}_n^{\sigma} \cap  L^{\frac{n}{2},\infty}  $ satisfies Assumption \ref{assumption_1}, then there exists $ C > 0 $ such that
\begin{align}
\label{theoremevariante_1}
 \big| \big| (H-z)^{-1} f \big| \big|_{L^{\frac{2n}{n-2},2}(\R^n)} \leq C ||f||_{L^{\frac{2n}{n+2},2}(\R^n)},\quad
 z \in \C \setminus [0,\infty),\  
 f \in L^2(\R^n) \cap  L^{\frac{2n}{n+2},2}(\R^n).
\end{align}
\end{theorem}

This theorem means essentially that the resolvent $ (H-z)^{-1} $ is uniformly bounded (in $z$) between $ L^{\frac{2n}{n+2},2} $ and $ L^{\frac{2n}{n-2},2} $ but, similarly to Theorem \ref{theorem_0}, we state it as above to make a clear distinction between the resolvent $ (H-z)^{-1} $ (defined on $ L^2 $) and its closure to $ L^{\frac{2n}{n-2},2} $. A similar remark also holds for Theorem \ref{theorem_1} below. The additional condition $ V \in L^{\frac{n}{2},\infty}$ is due to the use of the fact that the multiplication by $|V|^{1/2}$ is bounded from $L^{\frac{2n}{n-2},2}$ to $L^2$
, which allows us to deduce \eqref{theoremevariante_1} from weighted estimates in Theorem \ref{theorem_0} (1) and a perturbation method in Section \ref{sectionabstraite}.  Note that the norm in $ L^{\frac{n}{2},\infty} $ is also invariant by the scaling (\ref{scalingofpotentials}). 

\begin{remark}
Theorem \ref{theoremevariante} extends a part of the result by Kenig-Ruiz-Sogge \cite{KRS} for constant coefficient operators to Schr\"odinger operators with potentials. Extending such uniform Sobolev estimates to variable coefficients operators is a topic of current interest. Recently Guillarmou-Hassell \cite{GuHa} extended such estimates to the Laplace operator on non trapping asymptotically conic manifolds, and Hassell-Zhang \cite{HaZh} extended it to potential perturbations with smooth potentials decaying at infinity like $ \langle x \rangle^{-3} $ and without $0$ resonance nor eigenvalue. Here we provide a similar result on $ {\mathbb R}^n$  for potentials with critical singularity and weaker decay at infinity.

\end{remark}


To state our results on Strichartz inequalities, we recall the following classical definition.
\begin{definition}						
\label{admissible}
A pair $(p,q)$ is said to be an ($n$-dimensional) admissible pair if 
$$
2\le p,q\le \infty,\quad\frac2p=n\Big(\frac12-\frac1q\Big),\quad (n,p,q)\neq(2,2,\infty). 
$$
\end{definition}

\begin{theorem}[Global Strichartz estimates]	$ $
\label{theorem_1} 

\smallskip

\noindent {\rm (1)} Let $n\ge3$, $ \frac{n-1}{2} < \sigma \leq \frac{n}{2} $ and $V \in {\mathcal X}_n^{\sigma}$ satisfy Assumption \ref{assumption_1}. Then, for any admissible pairs $(p,q)$ and $(\tilde p,\tilde q)$ with $p,\tilde p>2$, there exists $ C > 0 $ such that  the solution $u$ to \eqref{equation_1} satisfies
\begin{align}
\label{theorem_1_1}
||u||_{L^{p} ( {\mathbb R} ; L^{q,2}(\R^n) )}
\le C ||\psi||_{L^2(\R^n)}+C ||F||_{L^{\tilde{p}^{\prime}} ( {\mathbb R} ; L^{\tilde{q}^{\prime},2}(\R^n) )},
\end{align}
for all $ \psi \in L^2(\R^n) $ and $ F \in L^1_{\rm loc} ({\mathbb R};L^2(\R^n)) \cap L^{\tilde{p}^{\prime}} ( {\mathbb R} ; L^{\tilde{q}^{\prime},2}(\R^n) )  $. Furthermore, if in addition $V\in L^{\frac n2,\infty}$, then \eqref{theorem_1_1} holds for all admissible pairs including the endpoint cases. \\
\noindent {\rm (2)} If $n=2$ and $ V \in {\mathcal X}_2 \cap L^1  $ satisfies Assumption \ref{assumption_2} then, for any admissible pairs $(p,q)$ and $(\tilde{p},\tilde{q})$, there exists $ C >0 $ such that 
\begin{align*}
||u||_{ L^{p} ( {\mathbb R}; L^{q,2}(\R^2))}
\le C||\psi||_{L^2(\R^2)}+C ||F||_{L^{\tilde{p}^{\prime}} ( {\mathbb R} ; L^{\tilde{q}^{\prime},2}(\R^2)) }
\end{align*}
for any $\psi \in L^2(\R^2)$ and $F \in L^1_{\rm loc} ({\mathbb R};L^2(\R^2) ) \cap  L^{\tilde{p}^{\prime}} ( {\mathbb R} ; L^{\tilde{q}^{\prime}}(\R^2) )$. 

\end{theorem}

Technically, the additional condition $V\in L^1$ in the two dimensional case  is due to the fact that $\langle x \rangle^{-1}$ is not $-\Delta$-smooth, see after Proposition \ref{lemma_resolvent_3}. Note also that we take $ F \in L^1_{\rm loc} ({\mathbb R};L^2) $ to make sure that $ \Gamma_H F $ has a clear sense; of course, the above Strichartz estimates show that $ \Gamma_H $ has a bounded closure as an operator between $ L^{\tilde{p}^{\prime}} ({\mathbb R};L^{\tilde{q}^{\prime},2}) $ and $ L^{p} ({\mathbb R};L^{q,2})  $ if $n=3$, or  $ L^{\tilde{p}^{\prime}} ({\mathbb R};L^{\tilde{q}^{\prime}}) $ and $ L^{p} ({\mathbb R};L^{q})  $ if $n=2$. 
Using the continuous embeddings $ L^{q,2} \subset L^q $ and $ L^{\tilde{q}^{\prime},2} \subset L^{\tilde{q}^{\prime}} $ (see Paragraph \ref{sectionnotationLorentz}), we see that (\ref{theorem_1_1}) allows to recover the usual Strichartz estimates
 \begin{align}
 \nonumber
||u||_{L^{p} ( {\mathbb R} ; L^{q}(\R^n) )}
\le C ||\psi||_{L^2}+C ||F||_{L^{\tilde{p}^{\prime}} ( {\mathbb R} ; L^{\tilde{q}^{\prime}}(\R^n) )},
\end{align}
for $\psi\in L^2(\R^n)$ and $F\in L^1_{\rm loc} ({\mathbb R};L^2(\R^n)) \cap L^{\tilde{p}^{\prime}} ( {\mathbb R} ; L^{\tilde{q}^{\prime}}(\R^n) )  $. 

When $n\ge3$, we can also add a small scaling critical potential.

\begin{corollary}\label{corollary_2}
Let $n\ge3$, $ \frac{n-1}{2} < \sigma \leq \frac{n}{2} $ and $V_1 \in {\mathcal X}_n^{\sigma}$ satisfy Assumption \ref{assumption_1}. Let $V_2$ be real-valued such that $||V_2||_{M^{\frac n2,\sigma}}$ is sufficiently small. Then the solution $u$  to \eqref{equation_1} with $H=-\Delta+V_1+V_2$ satisfies \eqref{theorem_1_1} for all admissible pairs $(p,q)$ and $(\tilde p,\tilde q)$ with $p,\tilde p>2$. Moreover, if in addition $||V_2||_{L^{\frac n2,\infty}}$ is small enough and $V_1\in L^{\frac n2,\infty}$ then  \eqref{theorem_1_1} holds for  all admissible pairs including the endpoint cases. 
\end{corollary}

\begin{remark}
When $n\ge3$, Theorem \ref{theorem_1} (1) and Corollary \ref{corollary_2} cover all admissible cases including the inhomogeneous endpoint case $(p,q,\tilde p,\tilde q)=(2,\frac{2n}{n-2},2\frac{2n}{n-2})$, while previous literatures \cite{RoSc,BPST1,BPST2} considered homogeneous estimates only. Here recall that, for non-endpoint admissible pairs, the inhomogeneous estimates follow from the homogeneous estimates and the Christ-Kiselev lemma (see Appendix \ref{Christ_Kiselev}), but this is not the case for the endpoint estimate. Also note that \cite{BaDu,MMT} proved Strichartz estimates for all admissible pairs, but only for bounded potentials so that $V=o(\langle x\rangle^{-2}\langle\log x\rangle^{-2})$ or $C^1$ potentials satisfying $\partial_x^\alpha V=O(\langle x\rangle ^{-2-|\alpha|})$, $|\alpha|\le1$. Similarly to Theorem \ref{theorem_0}, our assumption allows one strong singularity and multiple weak singularities. We also refer to a recent result \cite{KiKo} which studied the non-endpoint estimates for small $V\in M^{\frac n2,\sigma}$ and the homogeneous endpoint estimate for small $V\in L^{\frac n2}$. Compared with this result, the novelty of Corollary \ref{corollary_2} is again the inhomogeneous endpoint estimate for small $V_2\in L^{\frac n2,\infty}$. 

When $n=2$, \cite{BPST1,FFFP}  considered a class of scaling invariant potentials of the form $V(x)=a(\theta)r^{-2}$ with $a(\theta)>0$. Although we impose a slightly stronger condition such as $V\in L^1(\R^2)$, we do not require such a symmetry. Moreover, methods in \cite{BPST1,FFFP} essentially rely on the explicit formula of the kernel of $e^{-itH}$ and it seems to be difficult to extend them to potentials which are not invariant under the scaling $V(x)\mapsto \lambda^2 V(\lambda x)$. 
\end{remark}

\begin{remark}
If we take $\delta_0=0$ in Assumption \ref{assumption_1}, the above results for $n\ge3$ do not hold in general. For instance, endpoint Strichartz estimates can fail in the case of $V(x)=-\frac{(n-2)^2}{4|x|^2}$. We refer to a subsequent work \cite{Miz2} for more details. 
\end{remark}

\begin{remark}
As in Corollary \ref{corollary_2}, Theorem \ref{theorem_0} (resp. \ref{theoremevariante}) still holds if we add a small potential $V_2\in M^{\frac n2,\sigma}$ (resp. $V_2\in L^{\frac n2,\infty}$) to the operator $H$. This observation will be essentially proved in the proof of Corollary \ref{corollary_2} in Section \ref{proof_corollary_2}. 
\end{remark}

\begin{remark}At a formal level, the proof of Theorems \ref{theorem_0}, \ref{theoremevariante_1} and \ref{theorem_1} are very simple and based on the following iterated resolvent and Duhamel identities:\begin{align*}R(z)&=R_0(z)-R_0(z)VR_0(z)+R_0(z)VR(z)VR_0(z),\\\Gamma_{H}&=\Gamma_{-\Delta}-i\Gamma_{-\Delta}V\Gamma_{-\Delta}-\Gamma_{-\Delta}V\Gamma_{H}V\Gamma_{-\Delta},\end{align*}which can be seen, at least formally, by applying usual resolvent or Duhamel identities twice, where $R(z)=(H-z)^{-1}$ and $R_0(z)=(-\Delta-z)^{-1}$. In the case of resolvent estimates for instance, this resolvent formula, together with the decomposition $V=|x|^{-1}\cdot |x|V$, allows us to deduce desired estimates for $R(z)$ from estimates for free resolvents $R_0(z)$, $R_0(z)|x|V$, $|x|^{-1}R_0(z)$ and the estimate for the weighted resolvent $|x|^{-1}R(z)|x|^{-1}$. The estimates for the free resolvents can be proved by using the explicit formula of $R_0(z)$, while the proof of the estimate of $|x|^{-1}R(z)|x|^{-1}$ relies on a multiplier technique by \cite{BVZ}. A rough strategy for  the proof of Strichartz estimates is similar. 

We however stress that, due to a strong singularity of $V$ at the origin, justifying the above formulas is not so obvious. In Section \ref{sectionabstraite}, we develop, in a quite abstract setting, such a perturbative technique with a rigorous justification of the above observation. \end{remark}


In the following last result we consider a different kind of assumption on the potentials. They can be of long range type, but locally we allow them to have a critical singularity which scales as our previous potentials (i.e. typically as $ |x|^{-2} $).

\begin{assumption} \label{assumptionalternative} $ $\\
{\rm(1)} If $ n \geq 3 $, $(x \cdot \nabla)^\ell V\in L^{\frac n2,\infty}_{\rm loc}(\R^n)$ for $\ell=0,1,2$. \\\ \ \ \ \,\,If $n=2$, $V=V_1+V_2$ with  $(x \cdot \nabla)^\ell V_1\in L^1_{\rm loc}(\R^2)$ and $|x|^2(x \cdot \nabla)^\ell V_2\in L^\infty_{\rm loc}(\R^2)$ for $\ell=0,1,2$.\\
{\rm(2)} There exists $ \delta_0 > 0 $ such that for all $ f \in C_0^{\infty} ({\mathbb R}^n \setminus 0) $, 
\begin{align*}
\langle (-\Delta + V) f ,  f \rangle \geq  \delta_0 ||\nabla f||_{L^2}^2 .
\end{align*}
{\rm(3)} There exists $ \delta_0 > 0 $ such that for all $f \in C_0^{\infty} ({\mathbb R}^n \setminus 0)$,
$$ \langle (- \Delta - x \cdot \nabla V ) f , f \rangle \geq \delta_0   \times \begin{cases}  ||\nabla f||_{L^2}^2 & \mbox{if} \ n \geq 3 \\
|| \nabla f ||_{L^2}^2 + || w f ||_{L^2}^2 & \mbox{if} \ n = 2  \end{cases} $$
for some smooth positive function $ w  $ on $ {\mathbb R}^2 \setminus 0 $ such that $ |x|w \in L^{\infty} $.\\
{\rm(4)} There exists $C>0$ such that for any $f\in C_0^{\infty} (\R^n \setminus 0) $
\begin{align*}
| \langle (2V+ x \cdot \nabla V)f,f \rangle |&\le C \langle (-\Delta+V+1)f,f \rangle,\\
| \langle (2 x \cdot \nabla V+(x \cdot \nabla)^2 V)f,f \rangle |&\le C \langle (-2\Delta- x \cdot \nabla V)f,f \rangle. 
\end{align*}
\end{assumption}
In (2), (3) and (4), all brackets are understood in the form sense (see paragraph \ref{realizations}).

\begin{example}
Assumption \ref{assumptionalternative} is satisfied by any $V\in C^2(\R^n\setminus\{0\})$ such that, for some $\mu\in(0,2]$,
\begin{itemize}
\item{if $ n \geq  3 $, there exist constants $ c_0 < \frac{(n-2)^2}{4} $ and $ C > 0 $ such that
\begin{align}
V(x)\ge - c_0 |x|^{-2},\quad - x \cdot \nabla V(x)\ge C^{-1} \langle x \rangle^{-\mu}-2 c_0 |x|^{-2},\quad |(x \cdot \nabla)^2V(x)|\le C |x|^{-\mu}. \nonumber
\end{align}}
\item{ If $ n \geq 2 $, there exist constants $ c_1, c_2 , C > 0 $ such that
\begin{align}
V(x)\ge  c_1 |x|^{-\mu},\quad - x \cdot \nabla V(x)\ge c_2 |x|^{-\mu},\quad |(x \cdot \nabla)^2V(x)|\le C |x|^{-\mu}. \nonumber
\end{align}
In this case, (3) holds with $ w (x) = |x|^{- \frac{\mu}{2}} \langle x \rangle^{-1 + \frac{\mu}{2}}  $. }
\end{itemize}
\end{example}



\begin{theorem}	
\label{theorem_3} Let $V$ satisfies Assumption \ref{assumptionalternative} and $ H = - \Delta + V $. \\
{\rm (1)} If $ n \geq 3 $, then
$$  \big| \big|  |x|^{-1} (H-z)^{-1} |x|^{-1} f \big| \big|_{L^2(\R^n)}  \leq C || f ||_{L^2(\R^n)}, \quad  z \in \C \setminus\mathbb \R,\ f \in C_0^{\infty} ({\mathbb R}^n \setminus 0) . $$
{\rm (2)} If $ n = 2 $, then, for $w$ in Assumption \ref{assumptionalternative} (3),
$$  \big| \big|  w (H-z)^{-1} w f \big| \big|_{L^2(\R^2)}  \leq C || f ||_{L^2(\R^2)}, \quad z \in \C \setminus\mathbb \R,\ f \in C_0^{\infty} ({\mathbb R}^2 \setminus 0) . $$
\end{theorem}

\begin{remark}
The proof of this theorem is quite different from that of Theorem \ref{theorem_0} and based on a version of Mourre's theory. 
The brief outline is as follows (see Section \ref{sectionweakMourre} for details): the commutator $S:=[H,iA]=-\Delta-x\cdot\nabla V$ is positive and the double commutator $[H,S]$ satisfies $-S\lesssim [H,S]\lesssim S$ by Assumption \ref{assumptionalternative} (3) and (4), respectively, where $A=-i(x\cdot\nabla+\nabla\cdot x)/2$ is the generator of the dilation.  Having in mind that (trivial) strict Mourre's estimate $S\ge (S^{1/2})^2$ holds without any spectral localization, one can show by means of Mourre's differential inequality technique (in this step a careful justification of routine arguments will be required due to the strong singularity of $V$) that, for a large constant $\kappa>1$, the operator $$S^{1/2}(A+i\kappa)^{-1}(H-z)^{-1}(A-i\kappa)^{-1}(S^{1/2})^*$$ is bounded on $L^2$ uniformly in $z\notin\R$. This uniform bound, together with Hardy's inequality if $n\ge3$ or Assumption \ref{assumptionalternative} (3) itself if $n=2$, yields the assertion. 
\end{remark}
As a consequence, we obtain weighted $ L^2 $ space-time estimates.
\begin{corollary} \label{Corollaire-Mourre}
 Let $V$ satisfies Assumption \ref{assumptionalternative}. Let $u$ be given by \eqref{equation_2}. \\
{\rm (1)} If $ n \geq 3 $, then there exists  $ C > 0 $ such that
$$  \big| \big|  |x|^{-1} u \big| \big|_{L^2({\mathbb R};L^2(\R^n))}  \leq C || \psi ||_{L^2(\R^n)} + C \big| \big| |x| F \big| \big|_{L^2 ({\mathbb R};L^2(\R^n))}, $$
for all $ \psi \in L^2(\R^n) $ and $ F \in L^1_{\rm loc} ({\mathbb R};L^2(\R^n)) $ such that $ |x|F \in L^2 ({\mathbb R};L^2(\R^n))  $.

\smallskip

\noindent {\rm (2)} If $ n = 2 $, then
then there exists  $ C > 0 $ such that
$$  \big| \big|  w u \big| \big|_{L^2({\mathbb R};L^2(\R^2))}  \leq C || \psi ||_{L^2(\R^2)} + C \big| \big| w^{-1} F \big| \big|_{L^2 ({\mathbb R};L^2(\R^2))}, $$
for all $ \psi \in L^2(\R^2) $ and $ F \in L^1_{\rm loc} ({\mathbb R};L^2(\R^2)) $ such that $ w^{-1} F \in L^2 ({\mathbb R};L^2(\R^2)) $.
\end{corollary}

\begin{remark}
Uniform resolvent estimates for long-range potentials decaying like $\langle x\rangle^{-\mu}$ at infinity (with less singularities than the present case) have been previously established by \cite{Nak,FoSk} with the usual smooth weight $\langle x\rangle^{-\rho}$ for some $\rho>1/2+\mu/4$. Our assumption does not require such a pointwise decaying condition at infinity. In passing, we also show that we can allow the singular weight $|x|^{-1}$, which (as already observed in Remark \ref{remark_theorem_0}) would be an important input in the application to Strichartz estimates for long-range singular potentials. 

Another closely related reference is the paper \cite{Richard}. Compared to this one, our main contribution is a simplification of the proof (which is closer to the original Mourre theory and does not use interpolation spaces); we also consider more singular potentials but, as mentioned by \cite{Richard} which formally only considered smooth potentials, one can expect the techniques of  \cite{Richard} to work as well for potentials similar to ours.
\end{remark}

The rest of the paper is organized as follows. In the  first part of the next section, we record several basic facts on some Function spaces used throughout the paper. The second part discusses the precise definition of our Schr\"odinger operator $H=-\Delta+V$ and its domain. Section \ref{sectionabstraite} is devoted to abstract perturbation methods which play a crucial role in the proof of main theorems. In Section \ref{free}, we collect several known results on uniform estimates for the free resolvent. In Section \ref{section_proof}, we prove the main theorems, except Theorem \ref{theorem_3} and Corollary \ref{Corollaire-Mourre}, by using materials prepared in Sections \ref{sectionabstraite} and \ref{free} and Appendix \ref{appendix_C}, while the proof of Theorem \ref{theorem_3} and Corollary \ref{Corollaire-Mourre} is given in Section \ref{sectionweakMourre}. In Appendix \ref{Christ_Kiselev}, we recall the Christ-Kiselev lemma which will be used several times in the paper. Finally, Appendix \ref{appendix_C} is devoted to the proof of Theorem \ref{theorem_resolvent_2main} on uniform resolvent estimates with the homogeneous weight $|x|^{-1}$. 

\bigskip

\noindent {\bf Acknowledgments.} It is a pleasure to thank Thomas Duyckaerts, Colin Guillarmou, Andrew Hassell and Nikolay Tzvetkov for sharing the (Duyckaerts) trick used in Proposition \ref{lemma_resolvent_1}. JMB is partially supported by ANR Grant GeRaSic, ANR-13-BS01-0007-01. HM is partially supported by JSPS Grant-in-Aid for Young Scientists (B), No. 25800083 and by Osaka University Research Abroad Program, No. 150S007.

\section{Preliminary materials}
\setcounter{equation}{0} 
\subsection{ Lorentz and Morrey-Campanato spaces} \label{sectionnotationLorentz}
Given a measure space $ (X,\mu) $ and indices $ 0 < q,\sigma \leq \infty  $, the Lorentz space $ L^{q,\sigma} $ is the set of measurable functions $f : X \rightarrow \C $ for which, if we let $ d_f (\alpha) = \mu( \{ x \ | \ |f(x)| > \alpha \} ) $ be the distribution function defined for $ \alpha \geq 0 $ and $ f^* (s) = \inf \{ \alpha > 0 \ | \ d_f (\alpha) \leq s \} $ the rearrangement defined for $ s > 0 $,   
$$ || f ||_{L^{q,\sigma}} :=   \big| \big|s^{\frac{1}{q} - \frac{1}{\sigma} } f^* \big| \big|_{L^{\sigma}((0,\infty),ds)} < \infty .  $$
Two functions of $ L^{q,\sigma} $ that coincide a.e. will be considered equal, as in usual Lebesgue spaces. We note in passing that $ L^{q,q} = L^q $ when $ q \geq 1 $.
The functional $ || \cdot ||_{L^{q,\sigma}} $ is in general not a norm (the triangle inequality fails). However, when $ 1 < q < \infty $ and $ 1 \leq \sigma \leq \infty $, there is a norm $ ||| \cdot |||_{L^{q,\sigma}} $ on $ L^{q,\sigma} $ for which $ L^{q,\sigma} $ is a Banach space and which is equivalent to $ || \cdot ||_{L^{q,\sigma}} $ in the sense that $ || f ||_{L^{q,\sigma}} \leq ||| f |||_{L^{q,\sigma}} \leq C(q,\sigma) || f ||_{L^{q,\sigma}} $ for some positive constant $ C (q,\sigma) $. Thus all continuity estimates for linear operators can be expressed in terms of $ || \cdot ||_{L^{q,\sigma}} $. The Lorentz spaces are non decreasing in $  \sigma $, i.e. $ L^{q, \sigma_1} \subset L^{q, \sigma_2} $ if $ \sigma_1 \leq \sigma_2 $, with continuous embeddings. If $ 1 \leq q, \sigma  \leq \infty $ and $ \frac{1}{q_1} + \frac{1}{q_2} = \frac{1}{q} $, $ \frac{1}{\sigma_1} + \frac{1}{\sigma_2} = \frac{1}{\sigma} $, one has  the H\"older inequality
\begin{align}
 ||f g ||_{L^{q,\sigma}}  \le C || f ||_{L^{q_1,\sigma_1}} ||g||_{L^{q_2 , \sigma_2 }}  . \label{HolderLorentz}
\end{align}
If $ 1 < q, \sigma<\infty $, if $  (X , \mu ) $ has no atoms and is sigma-finite, one has $ \big( L^{q,\sigma} \big)^* = L^{q^{\prime},\sigma^{\prime}} $ and
\begin{align}
 || g ||_{L^{q^{\prime},\sigma^{\prime}}} \approx \sup_{1=||f||_{L^{q,\sigma}} } \left| \int_X f g d \mu \right|   , \label{pourcalculnorme1}
\end{align} 
where $ \approx $ means that the quotient of the two sides (when $ g \ne 0 $) is bounded from above and below by constants independent of $g$ (see \cite[pp 54-55]{Grafakos}). Using that simple functions are dense in $ L^{q,\sigma} $, one may restrict $f$ to the set of simple functions in the above supremum. For $1<q_1,q_2,\sigma_1,\sigma_2<\infty$, it is also useful to recall that if there exist dense subsets $\mathcal D_1\subset L^{q_1,\sigma_1}$ and $\mathcal D_2\subset L^{q_2,\sigma_2}$ such that a linear operator $A$ satisfies 
$$
|\langle Af,g\rangle|\le C||f||_{L^{q_1,\sigma_1}}||g||_{L^{q_2,\sigma_2}},\quad f\in \mathcal D_1,\ g\in \mathcal D_2,
$$
where $\langle f,g\rangle:=\int_Xf\overline gd\mu$, then \eqref{pourcalculnorme1} and the fact $(L^{q,\sigma} \big)^* = L^{q^{\prime},\sigma^{\prime}}$ imply that 
$$
||Af||_{L^{q_2',\sigma_2'}}\le C||f||_{L^{q_1,\sigma_1}},\quad f\in\mathcal D_1,
$$
and thus $A$ has a bounded closure as an operator in  $\mathbb B(L^{q_1,\sigma_1},L^{q'_2,\sigma'_2})$. 

Since $ L^{q,\sigma} $ is a Banach space if $ 1 < q,\sigma < \infty $, the space $ L^2_T L^{q,\sigma} := L^2 \big( [-T,T] , L^{q,\sigma} \big) $  is a Banach space for the norm $ || F ||_{L^2_T L^{q,\sigma}} = ( \int_{-T}^T || F (t) ||_{L^{q,\sigma}}^2 dt )^{1/2} $ (defined by means of Bochner's integrals, see \cite{Arendt}). We denote the natural (sesquilinear) duality between $ L^2_T L^{q,\sigma} $ and $ L^2_T L^{q^{\prime},\sigma^{\prime}} $  by
\begin{align}
 \langle F , G \rangle_T = \int_{-T}^{T} \left( \int_{X} F (t)  \overline{G(t)} d \mu \right) dt .  \label{pairing}
\end{align}
Similarly to (\ref{pourcalculnorme1}), when $ (X,\mu) $ is sigma-finite with no atoms, one has (see \cite[Prop. 4.5.7]{Grafakos}),
\begin{align}
 || F ||_{L^2_T L^{q,\sigma}} \approx \sup_{1=||G||_{L^2_T L^{q^{\prime},\sigma^{\prime}}} } | \langle F , G \rangle_T|  . \label{pourcalculnorme2}
 \end{align}
 Furthermore, since $ L^{q,\sigma}(X) $ is reflexive it has the so-called Radon-Nikodym property, hence one has $ \big( L^2_T L^{q,\sigma} \big)^* = L_T^2 L^{q^{\prime},\sigma^{\prime}} $ (see \cite{Diestel}). As above, if there exist dense subsets $\mathcal D_j\subset L^2_TL^{q_j,\sigma_j}$ such that a linear operator $A$ satisfies 
 $$
| \langle AF,G \rangle_{T}|\le C||F||_{L^2_TL^{q_1,\sigma_1}}||G||_{L^2_TL^{q_2,\sigma_2}},\quad (F,G)\in \mathcal D_1\times \mathcal D_2,
 $$
  then \eqref{pourcalculnorme2} implies that $A$ extends to a bounded operator from $L^2_TL^{q_1,\sigma_1}$ to $L^2_TL^{q'_2,\sigma'_2}$. 

In the special case where $ X = {\mathbb R}^n $ with $ n \geq 3 $,  the Sobolev space $ {\mathcal H}^1 $ is contained in $ L^{2^*,2} $ and 
\begin{align}
     || f ||_{L^{2^*,2}} \le C || \nabla f ||_{L^2 }  , \label{SobolevLorentz}
 \end{align} 
 which is slightly more precise than the usual Sobolev inequality since $ L^{2^*,2} \subset L^{2^*} $ (see \cite{Talenti,Cassani}). Here and below, when $n\ge3$, we use the classical notation  $$ 2^* = \frac{2n}{n-2} ,\quad 2_* = \frac{2 n}{n+2}.$$

 We next recall basic results on Morrey-Campanato spaces. As Lebesgue and Lorentz spaces (on  $ {\mathbb R}^n $), they satisfy the H\"older inequality
 \begin{align}
  || f g ||_{M^{q_0,\sigma_0}} \leq ||f||_{M^{q_1,\sigma_1}} ||g ||_{M^{q_2,\sigma_2}} \label{HolderMorrey}
\end{align}  
 if $ 1 \leq \sigma_j  \leq q_j < \infty$, $ \frac{1}{p_0} = \frac{1}{p_1} + \frac{1}{p_2} $ and $ \frac{1}{\sigma_0} = \frac{1}{\sigma_1} + \frac{1}{\sigma_2} $. They are non increasing in $  \sigma $, i.e.
 $$ L^q = M^{q,q} \subset L^{q,\infty} \subset M^{q,\sigma_2} \subset M^{q,\sigma_1}, \quad 1 \leq \sigma_1 \leq \sigma_2 <  q . $$
 We also recall that $|x|^{-1}\notin L^q$ for any $q$, but
 \begin{align}
 |x|^{-1} \in L^{n,\infty} \cap M^{n,\sigma_0}, \quad 1 \leq \sigma_0 < n . \label{espacexmoins1}
 \end{align}
 We also have the following important estimate (see \cite[Corollary after Theorem 5 in Chap. II]{Fef}): if $ V \in M^{\frac{n}{2}, \sigma} $ for some $ \sigma >1  $, then
 \begin{align}
  \big| \big| |V|^{\frac{1}{2}} f \big| \big|_{L^2} \leq C || \nabla f ||_{L^2}, \quad f \in C_0^{\infty} ({\mathbb R}^n \setminus 0) . \label{Fefferman}
\end{align} 
 We will see the interest of this property in the next paragraph and in Appendix \ref{appendix_C}.

 \subsection{Self-adjoint realizations} \label{realizations}

We denote by $ {\mathcal H}^1 = \{ f \in L^2 (\R^n) \ | \ ||f||_{L^2}^2 + ||\nabla f||_{L^2}^2 < \infty \} $ the usual Sobolev space. 


Given a locally integrable function $V : \R^n \rightarrow \R$, $ n \geq 2 $, we define the sesquilinear form
\begin{align}
 Q_H (f,g) = \langle f  , (-\Delta + V) g \rangle := \langle \nabla f , \nabla g \rangle + \int V f \bar{g} dx  , \label{formequadratiquereutilisee}
\end{align}
first on $ C_0^{\infty} (\R^n \setminus 0) $. If it is nonnegative (as will always be the case in this paper), we let
$$ {\mathcal G}^1 = \mbox{closure of} \ C_0^{\infty} (\R^n \setminus 0) \ \mbox{for the norm} \ \big( ||f||^2_{L^2} + Q_H (f,f) \big)^{1/2} , $$
and still denote by $ Q_H $ the unique continuous extension of (\ref{formequadratiquereutilisee})  to $ {\mathcal G}^1 $. Defining (\ref{formequadratiquereutilisee}) on $ C_0^{\infty} (\R^n \setminus 0) $ rather than on $ C_0^{\infty}(\R^n) $  allows $V$ to have strong singularities at the origin, typically in dimension 2 where the Hardy inequality fails. We note however that when $ n \geq 3 $, $ C_0^{\infty} (\R^n \setminus 0) $ is dense in $  {\mathcal H}^1  $ so, if one knows additionally that $   (\delta-1)|| \nabla f||_{L^2}^2 \leq \langle V f ,f \rangle   \le C ||f||_{{\mathcal H}^1}^2   $  for some $ \delta > 0 $ and all $ f \in C_0^{\infty} (\R^n \setminus 0) $, then $ {\mathcal G}^1 = {\mathcal H}^1 $ (with equivalence of norms).

According to the assumptions (\ref{assumption_1_1}) and (\ref{positifdim2}), $ Q_H $ is nonnegative for the potentials considered in Theorems \ref{theorem_0}, \ref{theoremevariante} and \ref{theorem_1} (as well as in Corollaries \ref{corollary_theorem_0} and \ref{corollary_2}). In dimension $ n \geq 3 $, 
 the estimate (\ref{Fefferman}) implies that $ {\mathcal G}^1 = {\mathcal H}^1 $ since
\begin{align}
V = |x|^{-1} |x| V \in M^{\frac{n}{2},\sigma_1} \quad \mbox{for some} \  \frac{n-1}{2} < \sigma_1 < \frac{n}{2} \label{VMorrey}
\end{align}
by the H\"older inequality (\ref{HolderMorrey}) and (\ref{espacexmoins1}) (by choosing $ n - 1 < \sigma_0 < n $) together with the fact that $ |x| V \in M^{n,2\sigma} $.
  In dimension 2, we only have the continuous embedding $ {\mathcal G}^1 \subset {\mathcal H}^1 $. 
 
 In Theorem \ref{theorem_3}, we consider Assumption \ref{assumptionalternative} thanks to which the form (\ref{formequadratiquereutilisee}) is well defined on $ C_0^{\infty} ({\mathbb R}^n \setminus 0) $;  indeed, $  V f $ does not belong to $ L^2 $ in general, but the integral $ \int V f \bar{g} $ is well defined since $ V f \bar{g} \in L^{1} ({\mathbb R}^n) $ if $ f , g \in C_0^{\infty} ({\mathbb R}^n \setminus 0) $ (in dimension $n \geq 3$, $ L^{\frac{n}{2},\infty} \subset L^{1}_{\mathrm{loc}}  $ using the H\"older inequality (\ref{HolderLorentz}) and that characteristic functions of compact sets belong to  $ L^{{n}/{(n-2)},1} $). Thanks to (2), $ Q_H $ is nonnegative and $ {\mathcal G}^1 \subset {\mathcal H}^1 $.   We also record in passing that $ \langle (x \cdot \nabla V) f , g \rangle $ must be interpreted in the distributions sense  $ - \int V \nabla_x \cdot (x f \bar{g}) dx $. The same remark holds for $ (x\cdot \nabla V)^2 $. 
 

In all these cases, we can define the self-adjoint operator $ H : D (H) \rightarrow L^2 $ in the usual way: the domain is given by
$$ D (H) = \big\{ f \in {\mathcal G}^1 \ | \ | Q_H (f,g) | \leq C_f ||g||_{L^2} \ \mbox{for all} \ g \in {\mathcal G}^1 \big\} $$
and then $ H f $ is the unique element in $L^2$ such that $ Q_H (f,g) = \langle H f , g \rangle $ for all $g \in {\mathcal G}^1 $. $ D (H) $ is dense in $L^2$ and in $ {\mathcal G}^1 $. Furthermore, $ {\mathcal G}^1 $ is continuously embedded into the Sobolev space $ {\mathcal H}^1 $.



 \section{A method of smooth perturbations} \label{sectionabstraite}
 \setcounter{equation}{0}
Let $\mathcal H$ be a complex Hilbert space with inner product $\langle\cdot,\cdot\rangle$ and norm $||\cdot||$. Given two self-adjoint operators $ (H_0,D(H_0)) $ and $ (H,D(H)) $ on $\mathcal H$, we prepare abstract smooth perturbation techniques which enable us to deduce estimates between Banach spaces for the resolvent $(H-z)^{-1}$ or the evolution group $e^{-itH}$ of the perturbed Hamiltonian $H$ from corresponding estimates for the free Hamiltonian $H_0$ and weighted estimates for $(H-z)^{-1}$ or $e^{-itH}$ in Hilbert spaces. 

Throughout this section we assume that $H$ can be written as $ H = H_0 + Y^* Z $ for some densely defined closed operators $ (Y,D (Y)) $ and $ (Z,D(Z)) $ in the sense that
 \begin{align}
 \label{conditionsdomaines}
 &D (H_0) \cup D(H)\subset D (Y) \cap D (Z),\\
  &\langle  H f  , g \rangle 
  =   \langle f , H_0 g \rangle + \langle Z f , Y g \rangle \quad\text{for}\quad f \in D (H) , \ g \in D (H_0).   \label{symmetrierigorouseuse}
 \end{align} 
These conditions will be satisfied in our applications. Note that under these conditions, $Y^*,Z^*$ are also densely defined closed operators (see \cite[Theorem VIII.1]{Reed_Simon}) and hence $Y$ and $Z$ are both $H_0$- and $H$- bounded by the closed graph theorem. We denote the resolvents of $H_0$ and $H$ by 
\begin{align}
\label{resolvent_1}
R_{H_0} (z) = (H_0 - z)^{-1}:\mathcal H\to D(H_0), \quad R_H (z) = (H-z)^{-1}:\mathcal H\to D(H)
\end{align}
for $z\in \rho(H_0)\cap \rho(H)$, where, given a linear operator $A$, $\rho(A)$ denotes the resolvent set of $A$. 

Recall that a pair of two Banach spaces $(\mathcal A_1,\mathcal A_2)$ is said to be a Banach couple if $\mathcal A_1,\mathcal A_2$ are algebraically and topologically embedded in a Hausdorff topological vector space $\widetilde {\mathcal A}$. Note that $\mathcal A_1\cap \mathcal A_2$ is well defined in this case. Then our abstract result on resolvent estimates is as follows: 

 \begin{proposition}[Abstract resolvent estimates]
\label{proposition_abstract_1} Let $\mathcal A$ and $\mathcal B$ be two Banach spaces such that $(\mathcal H,\mathcal A)$ and $(\mathcal H,\mathcal B)$ are Banach couples. 
Suppose $z\in \rho(H_0)\cap \rho(H)$ and assume there exist positive constants $ r_1,...,r_5$  (possibly depending on $z$) such that 
\begin{align}
 \big| \langle R_{H_0}(z) \psi , \varphi \rangle \big| 
 &\le  r_1  ||\psi||_{\mathcal A} ||\varphi||_{\mathcal B},   \label{proposition_abstract_1_1} \\
   || Z R_{H_0} (z)\psi || 
   &\le r_2 || \psi ||_{\mathcal A} , \label{proposition_abstract_1_2}   \\
  ||Y R_{H_0} (z) \psi ||    
  &\le r_3  ||\psi||_{\mathcal A},\label{proposition_abstract_1_3} \\
  ||Y R_{H_0} (\overline z) \varphi ||  
  &\le r_4  ||\varphi||_{\mathcal B}\label{proposition_abstract_1_4}\\
 || Z R_H (\bar{z}) Z^* h ||   &\le r_5 ||h||  \label{proposition_abstract_1_5}
\end{align}
for all  $ \psi \in \mathcal H\cap{\mathcal A}$, $\varphi\in\mathcal H\cap \mathcal B$ and $h \in D (Z^*)$.
Then, for all $ \psi \in \mathcal H\cap{\mathcal A}$ and $\varphi\in\mathcal H\cap \mathcal B$,
\begin{align}
  \big| \langle R_H(z) \psi , \varphi \rangle \big|  \leq  \big(r_1+ r_2r_4+  r_3r_4 r_5 \big) || \psi ||_{\mathcal A}  || \varphi ||_{\mathcal B} .  \label{inegaliteresolvantefaible}
\end{align}  
\end{proposition}
Note that (\ref{conditionsdomaines}) and \eqref{resolvent_1} guarantee that the left hand sides of (\ref{proposition_abstract_1_2}) to (\ref{proposition_abstract_1_5}) are well defined.

\begin{remark}
As examples of $\mathcal H,\mathcal A$ and $\mathcal B$, we mainly have in mind that $\mathcal H=L^2(X)$ and $\mathcal A,\mathcal B$ are weighted $L^2$-spaces $w(x)L^2(X)$ or Lorentz spaces $L^{p,q}(X)$ on some non-atomic sigma-finite measure space $(X,\mu)$. When $X=\R^n$, $n\ge2$, one can also consider weighted Sobolev spaces $w(x)|\nabla|^{1/2}L^2(\R^n)$ as  examples of $\mathcal A,\mathcal B$. 
\end{remark}

\begin{proof}
The proof follows  from the resolvent identity, which can be written in our context as \begin{align}
 \langle R_H (z) u , v \rangle
  &= \langle R_{H_0} (z) u , v \rangle - \langle Z R_H (z) u , Y R_{H_0} (\bar{z}) v \rangle  \label{resolventeidentity1}\\
  &= \langle R_{H_0} (z) u , v \rangle - \langle Y R_{H_0} (z) u , Z R_H (\bar{z}) v \rangle . \label{resolventeidentity2}
\end{align} 
for $u, v \in \mathcal H$. \eqref{resolventeidentity1} follows from (\ref{symmetrierigorouseuse}) with $ f = R_H (z) u$ and $ g = R_{H_0} (\bar{z}) v $, while \eqref{resolventeidentity2} is verified by exchanging the roles of $u$ and $v$ in (\ref{resolventeidentity1}), replacing $z$ by $ \bar{z} $ and taking the complex conjugate. Let $ \psi \in \mathcal H\cap{\mathcal A}$ and $\varphi\in\mathcal H\cap \mathcal B$. By (\ref{resolventeidentity1}) and the Cauchy-Schwarz inequality, we have
\begin{align}
  \big| \langle R_H(z)\psi ,\varphi \rangle \big| &\le  \big| \langle R_{H_0}(z)\psi ,\varphi \rangle \big| + || Z R_H (z) \psi ||   ||Y R_{H_0} (\bar{z}) \varphi ||    \nonumber \\
  &\le r_1 || \psi||_{\mathcal A} ||\varphi||_{\mathcal B}  + r_4 ||\varphi||_{\mathcal B} || Z R_H (z) \psi ||    , \label{combiningdebut}
  \end{align}
the second line following from (\ref{proposition_abstract_1_1}) and  (\ref{proposition_abstract_1_4}).
It remains to estimate
$$ || Z R_H (z)\psi ||   = \sup_{ || h ||   = 1 } \big| \langle Z R_H (z)\psi , h  \rangle \big| ,
 $$
 where one can take $ h \in D (Z^*) $ since this domain is   dense in $\mathcal H$.
Using the other resolvent identity (\ref{resolventeidentity2}), the Cauchy Schwarz inequality, (\ref{proposition_abstract_1_2}), (\ref{proposition_abstract_1_3}) and (\ref{proposition_abstract_1_5}) we obtain
\begin{align}
 \big| \langle Z R_H (z) \psi , h  \rangle \big| &\le \big| \langle  R_{H_0} (z) \psi , Z^* h  \rangle \big| + \big| \langle Y R_{H_0} (z) \psi , Z R_H (\bar{z}) Z^* h  \rangle \big| \nonumber \\
 &\le r_2 || \psi ||_{\mathcal A} + r_3 r_5 ||\psi||_{\mathcal A} , \nonumber
\end{align} 
which together with (\ref{combiningdebut}) yield (\ref{inegaliteresolvantefaible}). 
\end{proof} 

Next we consider abstract methods to derive space-time inequalities for Schr\"odinger equations. Proposition \ref{proposition_abstract_1} follows mainly from the resolvent identity which, in our abstract framework, is written in weak form (see (\ref{resolventeidentity1}) and (\ref{resolventeidentity2})).  Similarly, the proof of Theorem \ref{theorem_inhomogeneous_1} below uses weak forms of the Duhamel formula (see Proposition \ref{prop-weakDuhamel}). Stating them rigorously requires some care and a preparatory discussion since neither $Y$ nor $Z$ are assumed to be bounded on $\mathcal H$.
  
 Let us  recall the notion of $H$-(super)smoothness in the sense of Kato \cite{Kato} and Kato-Yajima \cite{KY}.
 A densely defined closed  operator $B : D (B) \rightarrow \mathcal H  $ is  $H$-smooth (with  bound $a$) if 
\begin{align}
\label{H-smooth_1}
\sup_{z\in \C \setminus \R}| \langle (R_H(z)-R_H(\overline z))B^*f,B^*f \rangle |\le \frac{a^2}{2}||f||  ^2,\quad f\in D(B^*). 
\end{align}
This is equivalent (see \cite[Theorem XIII. 25]{Reed_Simon}) to the fact that
, for any $ \psi \in \mathcal H $, $e^{-itH} \psi $ belongs to $ D(B)$ for  a.e. $t\in\R$ and
\begin{align}
\label{H-smooth_2}
||B e^{-itH} \psi||_{L^2_t\mathcal H} := \Big( \int_{\R}|| Be^{-itH}\psi ||  ^2 d t \Big)^{1/2}\le a ||\psi||  ,\quad \psi\in \mathcal H.
\end{align}
In particular, $B e^{-itH} \psi\in L^2([-T,T];\mathcal H)\subset L^1([-T,T];\mathcal H)$ for any $T>0$.  We also recall that if $B$ is $H$-smooth then $ D (H) \subset D (B) $ and $B$ is $H$-bounded with relative bound $ 0 $ (see Theorem XIII.22  of \cite{Reed_Simon}). 
A densely defined closed  operator $B$ is called $H$-supersmooth (with  bound $a$) if 
\begin{align}
\label{H-supersmooth_1}
\sup_{z\in \C \setminus \R}| \langle R_H(z)B^*f,B^*f \rangle|\le \frac a2 || f ||  ^2 , \quad f \in D (B^*) .
\end{align}
For instance the assumption (\ref{proposition_abstract_1_5})  is satisfied if $ Z $ is $H$-supersmooth with bound $ 2 r_5 $; note however that the assumptions of Proposition \ref{proposition_abstract_1} hold for a {\it single} $z$, not {\it all} $z \in \rho(H_0)\cap\rho(H)$. Also note  that if $B$ is $H$-supersmooth with bound $a$ then $B$ is $H$-smooth with bound $\sqrt{2a}$. 

The supersmoothness property (\ref{H-supersmooth_1}) implies that, for any simple function $ F : \R \rightarrow D (B^*) $, $Be^{-i(t-s)H}B^*F(s)$ is Bochner integrable over $s\in[0,t]$ (or $[t,0]$) and that for any $ T > 0 $,
\begin{align}
\label{H_supersmooth_2}
\left| \left| \int_0^tBe^{-i(t-s)H}B^*F(s)ds\right| \right|_{L^2_T\mathcal H}
\le a || F||_{L^2_T\mathcal H},
\end{align}
(see  \cite[Theorem 2.4]{Dan}). 
Here and below, we use the notation $L^p_T \mathcal B := L^p ( [-T,T] , \mathcal B )$. 

Consider the Duhamel operator $ \Gamma_H : L^1_T \mathcal H \rightarrow C_T \mathcal H := C ([-T,T], \mathcal H) $  defined by
$$ \Gamma_H F (t) = \int_0^t e^{-i(t-s) H} F (s) ds  $$
using the Bochner integral. It is not hard to check that one has
\begin{align}
\label{H_supersmooth_3}
\langle \Gamma_H F , G \rangle_T = \langle F , \Gamma_H^* G \rangle_T, \quad F , G \in L^1_T\mathcal H,
\end{align}
where $\langle F,G\rangle_T:=\int_{-T}^T\langle F,G\rangle dt$ and 
\begin{align}
 ( \Gamma_H^* G ) (t) &= 
  {\mathds 1}_{\R^+} (t) \int_t^{T} e^{-i(t-s)H} G (s)ds 
  -
  {\mathds 1}_{\R^-} (t) \int_{-T}^t e^{-i(t-s)H} G (s)ds 
  . \label{adjointDuhamel}
\end{align}

The following lemma gives the precise meaning of the operators $B\Gamma_H$ and $B\Gamma_H^*$: 
\begin{lemma}  \label{lemmeavecchi} Assume that $ B$ is 
$ H $-smooth with bound $a$. Let $ \chi \in C_0^{\infty}(\R)$ be such that $ \chi \equiv 1 $ near $ 0 $ and $0\le \chi\le1$. Then, the strong limits $$   \overline{B \Gamma_H } := \slim_{\epsilon \rightarrow 0} B  \chi (\epsilon H) \Gamma_H , \quad  \overline{B \Gamma_H^*} := \slim_{\epsilon \rightarrow 0} B  \chi (\epsilon H) \Gamma_H^* $$ exist in $ L^2_T\mathcal H $ and satisfy, uniformly in $ T > 0 $, $$ \big| \big| \overline{B \Gamma_H }  F\big| \big|_{L^2_T\mathcal H} \leq C a \big| \big| F \big| \big|_{L^1_T\mathcal H}, \quad \big| \big| \overline{B \Gamma_H^*} \big| \big|_{L^2_T\mathcal H} \leq C a \big| \big| F \big| \big|_{L^1_T\mathcal H} $$ for some universal constant $ C $. \end{lemma}
 
%
\begin{proof}
We treat only the case of $ \Gamma_H $, the one of $ \Gamma_H^* $ being similar in view of the expression (\ref{adjointDuhamel}). Note first that $ \chi (\epsilon H) $ commutes with $ \Gamma_H $ and that $ B \chi (\epsilon H) $ is bounded on $ L^2 $ since $B$ is $H$-bounded so $ B \Gamma_H \chi (\epsilon H) = B \chi (\epsilon H) \Gamma_H $ is well defined on $ L^1_T\mathcal H $. Since $ \sup |\chi| \leq 1 $, the $H$-smoothness implies $$ \left| \left| B e^{-it H} \int_{[-T,T]} e^{i s H} \chi (\epsilon H) F (s)ds \right| \right|_{L^2_T\mathcal H} \leq a || F ||_{L^1_T\mathcal H}.  $$The same upper bound also holds if we replace $ [-T,T] $ by $ [0,T] $ in the left hand side. Then, using the Christ-Kiselev Lemma (see Appendix \ref{Christ_Kiselev}), we can replace $ [-T,T] $ by $  [0,t]$ up to the multiplication of $a$ by some universal constant and obtain the uniform bound\begin{align}    ||  B \chi (\epsilon H)\Gamma_H F ||_{L^2_T\mathcal H} \leq C a ||F||_{L^1_T\mathcal H} . \label{borneuniformeutile} \end{align} If $F$ is of the form $ \chi (\epsilon _0 H) F_0 $ for some fixed $ \epsilon_0 $, then $  B \chi (\epsilon H)\Gamma_H  \chi (\epsilon _0 H) F_0  $ converges to $  B \Gamma_H  \chi (\epsilon _0 H) F_0   $ in $ L^2_T\mathcal H $ by dominated convergence. The density of the functions of the form $ \chi (\epsilon _0 H) F_0  $ in $ L^2_T\mathcal H $ together with (\ref{borneuniformeutile}) allow to prove by routine arguments that $ B \chi (\epsilon H)\Gamma_H F $ converges in $ L^2_T\mathcal H $ as $ \epsilon \rightarrow 0 $ for any $F$. This completes the proof. \end{proof}


\begin{proposition}[Weak Duhamel formula] \label{prop-weakDuhamel} Let $ b $ be a bounded Borel function on $\R$. Suppose that $ Y $ is $ H_0 $-smooth and $ Z b (H) $ is $H$-smooth. 
Then,  for all $T > 0 $, $ \psi \in \mathcal H $ and $ F,G\in L^1_T\mathcal H$, one has 
\begin{align}
  \label{weakduhamel}
 \langle  \Gamma_H b (H) F  , G  \rangle_T &=  \langle \Gamma_{H_0} b(H) F  , G \rangle_T-i \langle   \overline{Z b (H) \Gamma_H} F , \overline{Y \Gamma_{H_0}^*} G   \rangle_T,\\
 \label{weakduhamelbis}
 \langle e^{-itH} b(H) \psi , G \rangle_T &=  \langle e^{-itH_0} b (H) \psi , G \rangle_T -i \langle  Z b (H) e^{-itH}  \psi , \overline{Y \Gamma_{H_0}^*} G \rangle_T,\\
  \label{weakduhamel2}
 \langle  \Gamma_H F  , b (H) G \rangle_T &=  \langle \Gamma_{H_0} F  , b (H) G \rangle_T -i \langle   \overline{Y \Gamma_{H_0}} F , \overline{Z b (H) \Gamma_{H}^*} G   \rangle_T. 
\end{align}
\end{proposition}

\begin{proof}
Let $ \varphi \in D (H) $, $ \psi \in D (H_0) $ and let $ f (t) = e^{-i t H}\varphi $, $ g (t) = e^{-i t H_0} \psi $. Then, by (\ref{symmetrierigorouseuse}),
\begin{eqnarray}
 \frac{d}{dt} \langle f (t) , g (t) \rangle  = i \langle f (t) , H_0 g (t) \rangle - i \langle H f (t) , g (t) \rangle = - i \langle Z f (t) , Y g (t)  \rangle  \label{continuousintime}
\end{eqnarray}
 By integration between $ 0 $ and $ t $ and then substitution of $ \psi $ by $ e^{i t H_0} \theta $ with $ \theta \in D (H_0)$, we find
$$ \langle e^{-itH} \varphi , \theta \rangle - \langle e^{-itH_0} \varphi , \theta \rangle  = -i\int_0^t  \langle  Z e^{-irH} \varphi ,  Y e^{-i(r-t)H_0} \theta  \rangle dr . $$
Note that the integrand is continuous in $r$ hence integrable since (\ref{continuousintime}) is continuous in $t$. Changing $ t $ into $ t - s $ and then replacing 
 $ \varphi $ by $ F_{\epsilon} (s) $ and $ \theta $ by $ G_{\epsilon} (t)  $ with $F_{\epsilon} (s) = \chi (\epsilon H) F (s)$ and $G_{\epsilon} (t) = \chi (\epsilon H_0) G (t)$
(where $ \chi $ is as in Lemma \ref{lemmeavecchi}), we obtain by integration in $s$ between $0$ and $ t $,
$$ \langle  \Gamma_H F_{\epsilon} (t) , G_{\epsilon}(t) \rangle -  \langle \Gamma_{H_0} F_{\epsilon} (t) , G_{\epsilon}(t) \rangle = -i\int_0^t \left( \int_s^{t}  \langle  Z e^{-i(\tau-s)H} F_{\epsilon}(s) ,  Y e^{-i(\tau-t)H_0} G_{\epsilon}(t)  \rangle d \tau \right) ds . \nonumber
 $$
The iterate integral is well defined since  $ \langle e^{-i(t-s)H} F_{\epsilon}(s) , G_{\epsilon} (t) \rangle - \langle e^{-i(t-s)H_0} F_{\epsilon}(s) , G_{\epsilon} (t) \theta \rangle  $ can be integrated  in $s$ on $ [0,t] $ by \cite[Prop. 1.3.4 p. 24]{Arendt}. Then we wish to use the Fubini Theorem to prove the formally easy fact that
\begin{eqnarray}
  \langle  \Gamma_H F_{\epsilon} (t) , G_{\epsilon}(t) \rangle = \langle \Gamma_{H_0} F_{\epsilon} (t) , G_{\epsilon}(t) \rangle -i \int_0^t \langle  Z \Gamma_H F_{\epsilon} (\tau) ,   Y e^{-i(\tau-t)H_0} G_{\epsilon}(t)  \rangle d \tau . \label{premierFubini} 
\end{eqnarray}  
To do so, we need to justify that the map
$$ [0,t]^2 \ni (\tau,s) \mapsto {\mathds 1}_{[s,t]} (\tau) {\mathds 1}_{[0,t]} (s)  \langle Z e^{-i(\tau-s)H} F_{\epsilon}(s) ,  Y e^{-i(\tau-t)H_0} G_{\epsilon}(t)  \rangle   $$
is (measurable and) integrable for any given $ t  $, say $ t> 0$ the case $ t < 0 $ being similar. To prove the measurability, we write
$$   \langle  Z e^{-i(\tau-s)H} F_{\epsilon}(s) ,  Y e^{-i(\tau-t)H_0} G_{\epsilon}(t)  \rangle  = \langle   F (s) , \tilde{G}_{\epsilon}(\tau) \rangle $$
with
$ \tilde{G}_{\epsilon}(\tau) = e^{i(\tau-s) H} (Z \chi (\epsilon H))^* Y \chi (\epsilon H_0) e^{-i(\tau - t) H_0} G (t)$. 
Clearly, $ \tilde{G}_{\epsilon} $ is continuous on $ \R $. Since $F$ is measurable, by definition (see \cite[p. 6]{Arendt}), it can be approximated by simple functions. Thus $  \langle   F (s) , \tilde{G}_{\epsilon}(\tau) \rangle $ can be approached by simple functions and hence is  measurable.  The integrability follows from the estimate
$$  \big| \langle F (s) , \tilde{G}_{\epsilon}(\tau) \rangle \big| \leq  ||Z \chi( \epsilon H)||_{{\mathbb B} (\mathcal H)} || Y \chi ( \epsilon H_0) ||_{{\mathbb B}(\mathcal H)} || G (t) ||_{L^2 } || F_0 (s) ||_{\mathcal H}  $$
 whose right hand side is integrable in $ (s,\tau) $ on $ [0,t]^2 $. Therefore Fubini's Theorem can be used to derive (\ref{premierFubini}). Then, integrating (\ref{premierFubini}) in $t$, and using now the Fubini Theorem in $(t,\tau)$, we obtain
\begin{eqnarray} 
 \langle  \Gamma_H  F_{\epsilon}  , G_{\epsilon} \rangle_T =  \langle \Gamma_{H_0}  F_{\epsilon}  , G_{\epsilon} \rangle_T +  \langle  i Z \Gamma_H  F_{\epsilon} ,  Y \Gamma_{H_0}^* G_{\epsilon}   \rangle_T    \label{weakduhamelintermediaire}.
  \end{eqnarray}
This second application of the Fubini Theorem is justified in the same way as above by writing
$$  \langle  Z \Gamma_H F_{\epsilon} (\tau) ,   Y e^{-i(\tau-t)H_0} G_{\epsilon}(t)  \rangle =  \langle  e^{i (\tau - t) H_0} (Y \chi (\epsilon H_0))^* (Z \chi ( \epsilon H) ) \Gamma_H F (\tau) ,    G(t)  \rangle   $$
where the first factor of the bracket in the right hand side belongs to $ C ( [-T,T]^2_{t,\tau} , L^2_x ) $. Replacing $ F $ by $ b(H)F $ and letting $ \epsilon \rightarrow 0 $ in (\ref{weakduhamelintermediaire}), we obtain (\ref{weakduhamel}) by using Lemma \ref{lemmeavecchi}. The proofs of (\ref{weakduhamelbis}) and (\ref{weakduhamel2}) are similar (for (\ref{weakduhamel2}) we exchange the roles of $ H $ and $H_0$). 

\end{proof}

The following result clarifies the sense of the integral in the left hand side of (\ref{H_supersmooth_2}).

\begin{lemma} \label{sansHille} 
Let $B$ be $ H $-smooth and $ F : [-T,T] \to \mathcal H $ be a simple function. Then
$$ \overline{B \Gamma_H}F (t) = \int_0^t B e^{-i(t-s)H}F(s)ds$$
for almost every $t$. In particular, two side of this inequality coincide in $ L^2_T \mathcal H $.
\end{lemma}

\begin{proof}
Let us write $ F (s) = \sum_j \mathds{1}_{M_j}(s) f_j $ for some measurable sets $ M_j \subset [-T,T] $ and $  f_j \in \mathcal H $. Here $j$ runs over a finite set which we omit. It follows from Lemma \ref{lemmeavecchi} that $ \overline{B \Gamma_H} F  $ is the limit of $ B \chi (\epsilon_n H) F $ in $ L^2_T \mathcal H$, provided $ n \rightarrow \infty $. This implies that, by taking some subsequence $ \epsilon_{n_k} $, there exist a subset $ \mathcal N \subset [-T,T] $ of measure zero such that, for all $ t \in [-T,T] \setminus \mathcal N $,
$$  \big| \big| \overline{B \Gamma_H} F (t) - B \chi (\epsilon_{n_k} H) \Gamma_H F(t)\big| \big|_{\mathcal H} \rightarrow 0, \qquad k \rightarrow \infty . $$
To obtain the result, it suffices to show that, for all $ t \in [-T,T] \setminus \mathcal N $,
\begin{eqnarray}
 \left| \left| \int_0^t B e^{-i(t-s)H} F(s)ds - B \chi (\epsilon_{n_k} H) \Gamma_H F(t) \right| \right|_{\mathcal H} \rightarrow 0, \qquad k \rightarrow \infty . \label{normequisuffit}
\end{eqnarray}
To prove this and to clarify the sense of $\int_0^t B e^{-i(t-s)H} F(s)ds $, we use that $ B $ is $ H$ smooth hence that, for any $f \in \mathcal H $, $ e^{isH} f $ belongs to $ D (B) $ for a.e. $s$ and that 
$$ \int_{\mathbb R} || B e^{isH} f ||_{\mathcal H}^2 ds \le C || f ||_{\mathcal H}^2 . $$
Moreover, $ B e^{isH} f \in L^1_{\rm loc} ({\mathbb R},\mathcal H) $ by H\"older's inequality. Then using this and the H\"older inequality, the norm in (\ref{normequisuffit}) can be bounded by
$$ \sum_j  \left| \left| \int_{[0,t] \cap M_j} B e^{isH} (1-\chi (\epsilon_{n_k}H))e^{-itH}f_j ds \right| \right|_{\mathcal H} \le C \sum_j |t|^{\frac{1}{2}} || (1-\chi (\epsilon_{n_k}H))f_j ||_{\mathcal H} $$
where the right hand side goes to zero as $ k \rightarrow \infty $. This completes the proof. \end{proof}

This lemma implies the following equivalence which is useful to obtain the estimates of $\overline{B\Gamma_H}B^*$ or $\overline{B\Gamma_H^*}B^*$ from the $H$-supersmoothness of $B$.

\begin{corollary}
\label{corollary_supersmooth}
Assume that $B$ is $H$-smooth. 
Let $F:[-T,T]\to D(B^*)$ be a simple function. Then \eqref{H_supersmooth_2} holds if and only if one of the following estimates hold:
\begin{align}
\label{lemma_supermooth_1}
||\overline{B\Gamma_H}B^*F||_{L^2_T\mathcal H}\le a||F||_{L^2_T\mathcal H},\quad
||\overline{B\Gamma_H^*}B^*F||_{L^2_T\mathcal H}\le a||F||_{L^2_T\mathcal H}.
\end{align}
In particular, if in addition $B$ is $H$-supersmooth with bound $a$ then \eqref{lemma_supermooth_1} hold for all simple function $F:[-T,T]\to D(B^*)$. 
\end{corollary}

\begin{proof}
Since $B^*F(t)$ is a simple function in $t$ with values in $\mathcal H$, the equivalence between the first estimate in \eqref{lemma_supermooth_1} and \eqref{H_supersmooth_2} follows from Lemma \ref{sansHille}. Then the equivalence between the first and second estimates in \eqref{lemma_supermooth_1} can be seen from the relation \eqref{H_supersmooth_3} and Lemma \ref{lemmeavecchi}. 
\end{proof}

 We are now in a position to state our abstract result on Strichartz estimates: 
 
 \begin{theorem}[Abstract Strichartz estimates I. The endpoint case] \label{theorem_inhomogeneous_1} Let $\mathcal A$ and $\mathcal B$ be as in Proposition \ref{proposition_abstract_1} and $ b $ a bounded Borel function on $ \R $. Suppose that 
 $Y$ is $H_0$-smooth and that 
 there exist positive constants $s_1,...,s_7$ such that the following conditions {\rm(S1)} to {\rm(S4)} are satisfied for all $ \psi \in \mathcal H $, $f \in \mathcal H \cap {\mathcal A}$ and simple functions $ F:[-T,T]\to D(H)\cap\mathcal A$ and $ G:[-T,T]\to D(H_0)\cap\mathcal B$:
  \begin{itemize}
\item[{\rm(S1)}] Free endpoint Strichartz estimates: 
\begin{align}
\label{theorem_inhomogeneous_1_1}
 | \langle e^{-it H_0} \psi , G  \rangle_T | & \le s_1 || \psi ||   || G ||_{L^2_T {\mathcal B}} ,\\
\label{theorem_inhomogeneous_1_2}
 \big| \langle \Gamma_{H_0}F , G \rangle_T \big| &\le s_2 ||F ||_{L^2_T\mathcal A} ||G||_{L^2_T {\mathcal B}} .
\end{align}
\item[{\rm(S2)}] Free inhomogeneous smoothing estimates: 
\begin{align}
\label{theorem_inhomogeneous_1_3}
||\overline{Y \Gamma_{H_0}^*}G||_{L^2_t\mathcal H}&\le s_3 ||G||_{L^2_T {\mathcal B}},\\
\label{theorem_inhomogeneous_1_4}
||\overline{Y\Gamma_{H_0}}F ||_{L^2_T\mathcal H}&\le s_4 ||F||_{L^2_T {\mathcal A}},\\
\label{theorem_inhomogeneous_1_5}
||\overline{Z |b|^2(H)\Gamma_{H_0}}F||_{L^2_T\mathcal H}&\le s_5 || F||_{L^2_T {\mathcal A}}, 
\end{align}
where, in (\ref{theorem_inhomogeneous_1_5}), we assume  that $ |b|^2(H) D(H_0)\subset D(Z)$.
\item[{\rm(S3)}] $Zb(H)$ is $H$-supersmooth with bound $s_6$.
\item[{\rm(S4)}] Stability of $\mathcal A$ by $b(H)$: 
$ 
|||b|^2 (H)f \big||_{\mathcal A}\le s_7 || f||_{\mathcal A}.
$ 
\end{itemize}
 Then, for all $ \psi \in \mathcal H $, all simple functions $ F:[-T,T]\to D(H)\cap\mathcal A$ and $ G:[-T,T]\to D(H_0)\cap\mathcal B$,   \begin{align}
\big|\langle e^{-itH} b (H) \psi,G\rangle_T\big| &\le  \left( s_1||b||_{L^{\infty}} +  (2s_6)^{1/2} s_3 \right) || \psi ||_{\mathcal H }||G||_{L^2_T\mathcal B}  \label{theorem_inhomogeneous_1_6fort}, \\
\big|\langle \Gamma_H |b|^2 (H) F,G\rangle_T\big|  &\le  \left(s_2 s_7 + s_3 s_5 + s_3 s_4 s_6 \right) || F ||_{L^2_T {\mathcal A} }||G||_{L^2_T\mathcal B} . \label{theorem_inhomogeneous_1_7fort}
  \end{align}
 \end{theorem}
 
\begin{proof}
Using (\ref{weakduhamelbis}), (\ref{theorem_inhomogeneous_1_1}) and  (\ref{theorem_inhomogeneous_1_3}), 
we find
$$  |  \langle e^{-itH} b (H) \psi , G  \rangle_T |  \leq    \left( s_1||b||_{L^{\infty}}|| \psi ||_{ \mathcal H} + s_3 ||Z b (H) e^{-itH}   \psi||_{L^2_T\mathcal H} \right) || G ||_{L^2_T {\mathcal B}}  $$
so we deduce (\ref{theorem_inhomogeneous_1_6fort}) from the $ H$-smoothness of $ Z b (H) $ (see (\ref{H-smooth_2})).
To prove (\ref{theorem_inhomogeneous_1_7fort}), we start by using  (\ref{weakduhamel}) together with 
 (\ref{theorem_inhomogeneous_1_2}), (\ref{theorem_inhomogeneous_1_3}) and (S4) to obtain
\begin{align}
 |  \langle  \Gamma_H |b|^2 (H) F , G  \rangle_T | &\le  s_2  \big| \big| |b|^2(H)  F \big| \big|_{L^2_T {\mathcal A} } || G ||_{L^2_T {\mathcal B}} +  || \overline{Z |b|^2 (H)\Gamma_H } F ||_{L^2_T\mathcal H } || \overline{Y \Gamma_{H_0}^*} G ||_{L^2_T\mathcal H } \nonumber \\
 &\le  \left( s_2 s_7 ||  F ||_{L^2_T {\mathcal A} }  + s_3 ||\overline{Z |b|^2 (H) \Gamma_H} F ||_{L^2_T\mathcal H } \right) ||  G ||_{L^2_T {\mathcal B} } . \label{areciteralafin_1}
\end{align}
Note that $ Z |b|^2 (H)= Z b (H) \bar{b}(H) $ is $H$-smooth 
since $ Z b (H) $ is. To estimate the term $\overline{Z |b|^2(H)\Gamma_H} F$, taking a simple function $\widetilde G:[-T,T]\to D(Z^*)$ such that $ ||\widetilde{G}||_{L^2_T\mathcal H} = 1 $ and using \eqref{weakduhamel2}, we obtain
\begin{align*}
   \langle \overline{Z |b|^2(H) \Gamma_H} F , \widetilde{G} \rangle_T  &=   \langle  \Gamma_H  F , |b|^2(H) Z^* \widetilde{G} \rangle_T   \\
   &=   \langle   \Gamma_{H_0}    F ,  |b|^2(H) Z^* \widetilde{G} \rangle_T 
    -i \langle \overline{Y \Gamma_{H_0} }   F , \overline{Z b (H) \Gamma_H^*  }  b(H)^* Z^* \widetilde{G} \rangle_T. 
\end{align*} 
By virtue of  (\ref{theorem_inhomogeneous_1_5}), the first term in the right hand side can be estimated as
\begin{align}
\label{areciteralafin_2}
|\langle   \Gamma_{H_0}    F ,  |b|^2(H) Z^* \widetilde{G} \rangle_T|\le ||\overline{Z|b|^2(H)\Gamma_{H_0}}    F||_{L^2_T\mathcal H}\le s_5||F||_{L^2_T\mathcal A},
\end{align}
while we use (S3) and Corollary \ref{corollary_supersmooth} as well as \eqref{theorem_inhomogeneous_1_4} to deal with the second term as
\begin{align}
\label{areciteralafin_3}
|\langle i \overline{Y \Gamma_{H_0} }   F , \overline{Z b (H) \Gamma_H^*  }  b(H)^* Z^* \widetilde{G} \rangle_T|\le s_4 s_6 ||   F ||_{L^2_T {\mathcal A}}.
\end{align} 
By taking the supremum over all $ \widetilde{G} $, we have
$$
||\overline{Z |b|^2(H)\Gamma_H} F||_{L^2\mathcal H}\le (s_5+s_4s_6)||F||_{L^2_T\mathcal A}
$$
which together with (\ref{areciteralafin_1}) completes the proof of (\ref{theorem_inhomogeneous_1_7fort}).  
\end{proof}

\begin{remark}
As seen in the above proof, the $H$-smoothness of $Z$ is sufficient to prove only the homogeneous estimate \eqref{theorem_inhomogeneous_1_6fort}, while the $H$-supersmoothness of $Z$ is unnecessary. \eqref{theorem_inhomogeneous_1_4}, \eqref{theorem_inhomogeneous_1_5} and (S4) also have not been used in the proof of \eqref{theorem_inhomogeneous_1_6fort}. 
\end{remark}

\begin{remark}
In above abstract theorems, we only consider estimates for the sesquilinear forms. 
It is also possible to state a criterion in an abstract setting to obtain that $R_H(z)$, $e^{-itH}b(H)$ and $\Gamma_H |b|^2(H)$ have bounded closures as operators in $\mathbb B(\mathcal A,\mathcal B^*)$, $\mathbb B(\mathcal H,L^2_T\mathcal B^*)$ and $\mathbb B(L^2_T\mathcal A, L^2_T\mathcal B^*)$, respectively, from the corresponding statements for $R_{H_0}(z)$, $e^{-itH_0}$ and $\Gamma_{H_0}$, and assumptions \eqref{proposition_abstract_1_1} to \eqref{proposition_abstract_1_5} or (S1) to (S4), respectively. However, it requires additional assumptions on $\mathcal A$, $\mathcal B$ and their dual spaces such as the Radon-Nikodym property and the representation theorem of the duality paring, which makes the proof and the statement rather involved. 

On the other hand, in concrete applications, such a boundedness can be easily seen from \eqref{inegaliteresolvantefaible} (or \eqref{theorem_inhomogeneous_1_6fort} and \eqref{theorem_inhomogeneous_1_7fort}) and standard duality and density arguments (especially, materials recorded in Subsection \ref{sectionnotationLorentz} in the case of $\mathcal A=\mathcal B=L^{2_*,2}$). 
\end{remark}

For the non-endpoint case, we have the following abstract theorem, which is essentially due to \cite{RoSc} in the case when both $Y$ and $Z$ are bounded. Here we do not require such a boundedness. 

 \begin{theorem}[Abstract Strichartz estimates II. The non-endpoint case] \label{theorem_inhomogeneous_2} Let $\mathcal B$ and $b$ be as in Theorem \ref{theorem_inhomogeneous_2}. Assume there exist positive constants $s_1,s_2,s_3$ and $p>2$ such that the following {\rm(S1$'$)} to {\rm(S3$'$)} are satisfied for all $ \psi \in \mathcal H $ and all simple functions $ G:[-T,T]\to D(H_0)\cap \mathcal B$:
  \begin{itemize}
\item[{\rm(S1$'$)}] Free Strichartz estimates: 
$ | \langle e^{-it H_0} \psi , G  \rangle_T |  \le s_1 || \psi ||   || G ||_{L^{p'}_T {\mathcal B}}$, where $p'=p/(p-1)$. 
\item[{\rm(S2$'$)}] $Y$ is $H_0$-smooth with bound $s_2$. 
\item[{\rm(S3$'$)}] $Zb(H)$ is $H$-smooth with bound $s_3$. 
\end{itemize}
Then, for all $ \psi \in \mathcal H $ and all simple functions $ G:[-T,T]\to D(H_0)\cap \mathcal B$, 
 one has
  \begin{align}
\big|\langle e^{-itH} b (H) \psi,G\rangle_T\big| &\le  \left( s_1||b||_{L^{\infty}} +  2C_ps_1s_2s_3 \right) || \psi ||_{\mathcal H }||G||_{L^2_T\mathcal B}  \label{theorem_inhomogeneous_2_1},
  \end{align}
  where $C_p=2^{2/p}(1-2^{1/p-1/2})^{-1}$. 
 \end{theorem}

\begin{proof}
Let us first show that (S2$'$) implies, for any simple function $G:[-T,T]\to D(H_0)\cap \mathcal B$, 
\begin{align}
\label{theorem_inhomogeneous_2_2}
|||\overline{Y\Gamma_{H_0}^*} G||_{L^2_T\mathcal H}\le 2C_ps_1s_2 || G||_{L^{p'}_T\mathcal B}.
\end{align}
To this end, we let $\chi$ be as in Lemma \ref{lemmeavecchi} and consider the following two operators:
$$
K_\epsilon G(t):=Y\chi(\epsilon H_0)\Gamma_{H_0}^*G(t),\quad
\tilde K_\epsilon G(t):=Y\chi(\epsilon H_0)\int_{-T}^Te^{-i(t-s)H_0}G(s)ds. 
$$
where $Y\chi(\epsilon H_0)$ is bounded on $\mathcal H$ thanks to the relative $H_0$-boundedness of $Y$. Then (S2$'$) implies 
$$
||\tilde K_\epsilon G||_{L^2_T\mathcal H}=\Big|\Big| Ye^{-itH}\int_{-T}^T e^{isH_0}\chi(\epsilon H_0)G(s)ds\Big|\Big|_{L^2_T\mathcal H}\le s_2\Big|\Big| \int_{-T}^Te^{isH_0}\chi(\epsilon H_0)G(s)ds\Big|\Big|_{\mathcal H},
$$
where, by the duality argument and (S1$'$) as well as the fact $|\chi|\le1$, the right hand side reads
$$
s_2\Big|\Big| \int_{-T}^Te^{isH_0}\chi(\epsilon H_0)G(s)ds\Big|\Big|_{\mathcal H}=s_2\sup_{||\varphi||=1}|\langle G,e^{-isH_0}\chi(\epsilon H_0)\varphi\rangle_T|\le s_1s_2||G||_{L^{p'}_T\mathcal B}.
$$
Taking the formula \eqref{adjointDuhamel} of $\Gamma_{H_0}^*$ into account, we use the Christ-Kiselev lemma to obtain 
$$
||K_\epsilon G(t)||_{L^2_T\mathcal H}\le 2C_ps_1s_2|| G||_{L^{p'}_T\mathcal B},
$$
where we note that $p'<2$. This uniform bound in $\epsilon$, together with the fact that $\overline{Y\Gamma^*} G=\lim\limits_{\epsilon\to0} K_\epsilon G$, shows \eqref{theorem_inhomogeneous_2_2}. Now the assertion is a consequence of \eqref{weakduhamelbis}, (S1$'$), \eqref{theorem_inhomogeneous_2_2} and (S3$'$). 
\end{proof}

\section{Free resolvent estimates} 
\label{free}
\setcounter{equation}{0}

In this section, we collect several estimates on the free resolvent $ R_0 (z) = (-\Delta -z)^{-1} $ of the Laplacian $\Delta$ on $ \R^n $, $ n \geq 2 $.  The following estimate is a generalization to Lorentz spaces of a special case of uniform $L^p$ resolvent estimates, also called uniform Sobolev inequalities, due to \cite{KRS}. 

\begin{proposition}
\label{lemma_resolvent_1}
Let $n\ge3$. Then there exists $C>0$ such that for all $z\in\C \setminus[0,\infty)$
$$
|| R_0(z)f ||_{L^{2^*,2}}\le C ||f ||_{L^{2_*,2}},\quad f\in L^{2_*,2} \cap L^2. 
$$
\end{proposition}

\begin{proof}
The following proof is due to T. Duyckaerts  \cite{Duy2} 
(see also \cite[Remark 8.8]{HaZh}). We first show that there exists $C>0$ such that for all $z\in \C$,
\begin{align}
\label{proof_lemma_resolvent_1_1}
|| f ||_{L^{2^*,2}}\le C ||(-\Delta-z)f ||_{L^{2_*,2}}\quad f\in {\mathcal S}(\R^n),
\end{align}
where $ {\mathcal S}(\R^n)$ is the space of Schwartz functions. Let $u(t):=e^{izt}f$ which solves
$$
i\partial_tu=\Delta u+F,\quad u|_{t=0}=f ,
$$
where $F=e^{izt}(-\Delta-z)f$. Then the endpoint Strichartz estimate for the free Schr\"odinger equation in Lorentz spaces (see \cite[Theorem 10.1]{KT}) implies that for any $T>0$,
\begin{align}
\label{proof_lemma_resolvent_1_2}
|| u ||_{L^2_TL^{2^*,2}}\le C ||f||_{L^2}+C ||F||_{L^{2}_TL^{2_*,2}}, 
\end{align}
where $C$ is independent of $T$. By virtue of the specific formula of $u$ and $F$, one can compute
\begin{align*}
||u||_{L^2_T L^{2^*,2}}=\gamma(z,T)||f||_{L^{2^*,2}},\quad
|| F ||_{L^2_TL^{2_*,2}}= \gamma(z,T) ||(-\Delta-z)f||_{L^{2_*,2}},
\end{align*}
where $\gamma(z,T):= || e^{izt} ||_{L^2_T} \geq \sqrt{T} , $
since $ |e^{izt}| \geq 1 $ either on $ [0,T] $ or $ [-T,0] $. In particular, $\gamma(z,T)\to \infty$ as $T\to  \infty$ for each $z$, so dividing by $\gamma(z,T)$ and letting $T\to  \infty$ in \eqref{proof_lemma_resolvent_1_2}, we obtain \eqref{proof_lemma_resolvent_1_1}. 

Now we  show that \eqref{proof_lemma_resolvent_1_1} implies the assertion. For $z\in\C \setminus[0,\infty)$, $(-\Delta-z)^{-1}$ maps $ {\mathcal S}(\R^n)$ into itself so, by plugging $g=(-\Delta-z)^{-1}f$ with $f\in {\mathcal S}(\R^n)$ into \eqref{proof_lemma_resolvent_1_1}, we  obtain the assertion for $f \in {\mathcal S}(\R^n)$, which also implies the assertion for all $f\in L^2\cap L^{2^*,2}$ by the density. 
\end{proof}

We next record two results on weighted resolvent estimates. 

\begin{proposition}
\label{lemma_resolvent_3}
For any $w\in L^2(\R^2)$, the multiplication operator by $w$ is $-\Delta$-smooth. 
\end{proposition}

\begin{proof}
Using the characterization (\ref{H-smooth_2}), it is an immediate consequence of the estimate 
$$
|| e^{it\Delta}\psi ||_{L^{\infty}_xL^2_t}\le C ||\psi||_{L^2}
$$
proved in \cite[Theorem 3]{Rogers}, and the trivial inequality $||wf||_{L^2}\le ||w ||_{L^2}||f||_{L^\infty}$. \end{proof}

Note that $ \langle x \rangle^{-1}$ is known to be not $-\Delta$-smooth if $n=2$ (see \cite{Walther}) so $L^2(\R^2)$ cannot be replaced by $L^{2,\infty}(\R^2)$ in general,  in contrast to higher dimensions $n\ge3$. This is the main reason to take $V\in L^1(\R^2)$ in Theorem  \ref{theorem_1}. 
\begin{proposition} Let $ n \geq 3 $.
\label{proposition_application_3}
Let $\alpha_1,\alpha_2, \sigma$ satisfy $\frac{2n}{n+1}< \alpha_1, \alpha_2\le 2$ and $\frac{n-1}{\alpha_1 + \alpha_2 - 2}< \sigma \le \frac {n}{\alpha_M}$, where $\alpha_M=\max(\alpha_1,\alpha_2)$. Then there exists $C=C(n,\alpha_1,\alpha_2,\sigma)>0$ such that, for any $w_1\in M^{\frac{2n}{\alpha_1},2\sigma}$ and $w_2\in M^{\frac{2n}{\alpha_2},2\sigma}$ with $w_1,w_2>0$, any $z\in \C \setminus[0,\infty)$ and any $ \varphi , \psi \in C_0^{\infty} (\R^n) $,
\begin{align}
\label{proposition_application_3_1}
\left| \left\langle R_0 (z) w_1 \varphi , w_2 \psi \right\rangle  \right| \le C|z|^{-1+\frac{\alpha_1+\alpha_2}{4}} ||w_1 ||_{M^{\frac{2n}{\alpha_1},2\sigma}} ||w_2||_{M^{\frac{2n}{\alpha_2},2\sigma}} || \varphi ||_{L^2} || \psi ||_{L^2}  .
\end{align}
\end{proposition}

Note that $ w_1,w_2 \in L^{2\sigma}_{\rm loc}\subset L^2_{\rm loc} $ so the right hand side of (\ref{proposition_application_3_1}) has a clear sense. This proposition (and its proof) is a slight modification of \cite[Lemma 4]{Fra}, the change being that we allow $ w_1$ to be different from $ w_2 $. It turns out to be useful for the applications. We will need it in proof of Theorem \ref{theorem_0} in paragraph \ref{premiereapplication}; this is also useful to prove  eigenvalues estimates (see \cite{Miz1}).

\begin{proof}
Since $ || w(\lambda\cdot)||_{M^{\frac{2n}{\alpha},2\sigma}}=\lambda^{-\alpha/2} ||w ||_{M^{\frac{2n}{\alpha},2\sigma}}$, it  suffices to show \eqref{proposition_application_3_1}  for  $|z|=1$, $z\neq1$. We take $\psi,\varphi\in C_0^\infty(\R^n)$ and may assume $||\psi||_{L^2}=||\varphi||_{L^2}=1$. We wish to interpolate between the simple bound
$$ \left| \left\langle  (-\Delta-z)^{-it} w_1^{ {it}} \varphi , w_2^{- {it} } \psi \right\rangle  \right| \le C e^{C|t|}
, \quad t \in \R $$
 and the non trivial following one, for some suitable $ s \geq 1 $ to be found,
%
 %
\begin{align}
\label{proof_proposition_application_3_1_preuve}
\left| \left\langle  (-\Delta-z)^{-s-it} w_1^{{s+it}} \varphi,w_2^{{s-it}} \psi  \right\rangle \right|
\le Ce^{Ct^2}|| w_1 ||_{M^{\frac{2n}{\alpha_1},2\sigma}}^s
||w_2||_{M^{\frac{2n}{\alpha_2},2\sigma}}^s
,  \quad t \in \R . 
\end{align}
Using that $ s + i t \mapsto \langle (-\Delta-z)^{-s-it} w_1^{{s+it}}  \varphi , w_2^{{s-it}} \psi  \rangle  $
is holomorphic for $ s \in (0, \sigma) $ and continuous for $ s \in [ 0 , \sigma] $,  (\ref{proposition_application_3_1}) will  follow by interpolation.
Note that the upper bound $s \leq \sigma$  ensures $ w_1^{s},w_2^s\in L^2_{\rm loc} $. Let us prove (\ref{proof_proposition_application_3_1_preuve}).
 The first tool is the following pointwise bound on the kernel of $(-\Delta-z)^{-s-it}$
\begin{align}
\label{proof_proposition_application_3_6}
|(-\Delta-z)^{-s-it}(x-y)|\le C e^{Ct^2}|x-y|^{-\frac{n+1}{2}+s},
\end{align}
which holds for $\frac{n-1}{2} \leq s < \frac{n+1}{2}$, $ t \in \R $ and uniformly in $|z|=1$, $z\neq1$. This is seen from the explicit formula of the kernel in term of Bessel functions   (see \cite[Section 2 (2.21)-(2.25)]{KRS}). 
The second tool is  a weighted boundedness of the fractional integral operator $ I_{\beta} $ (the convolution with $|x|^{-n+\beta}$ for $0<\beta<n$). It is shown in \cite{SaWh} that if $w,v>0$ satisfy, for some $1 < p < \infty$, 
\begin{align}
\label{proof_proposition_application_3_3}
\sup_{x,r} \Big\{r^{\beta}\Big(r^{-n}\int_{B_r(x)} w(y)^pdy\Big)^\frac{1}{2p}\Big(r^{-n}\int_{B_r(x)} v(y)^{-p}dy\Big)^\frac{1}{2p}\Big\}\le C_p
\end{align}
 then there exists $C=C(n,\beta)>0$ independent of $w,v$ and $C_p$ such that
\begin{align}
\label{proof_proposition_application_3_4}
\big| \big| w^{\frac 12}I_{\beta}  \varphi \big| \big|_{L^2}\le CC_p \big| \big| v^{\frac12} \varphi \big| \big|_{L^2},\quad  \varphi \in C_0^{\infty}(\R^n). 
\end{align}
If $v=\tilde{w}^{-1}$, then the left hand side of \eqref{proof_proposition_application_3_3} is dominated by $ ||w^{1/2}||_{M^{\frac{2n}{\beta_2},2p}}||\tilde{w}^{1/2}||_{M^{\frac{2n}{\beta_1},2p}}$, provided  $2 \beta =\beta_1+\beta_2$ and $1<p\le n/\max(\beta_1,\beta_2)$ (this last condition is required since the second index of a Morrey-Campanato space cannot be smaller than the first one). Therefore, \eqref{proof_proposition_application_3_4} shows that
\begin{align}
\label{proof_proposition_application_3_5}
\big| \langle  w^{1/2} I_{\beta} \tilde{w}^{1/2}  \varphi , \psi \rangle \big|\le C ||w^{1/2}||_{M^{\frac{2n}{\beta_2},2p}}||\tilde{w}^{1/2}||_{M^{\frac{2n}{\beta_1},2p}}
\end{align}
with some $C=C(n,\beta_1,\beta_2,p)>0$  independent of $w,\tilde{w}, \varphi$ and $\psi$.
Now, using (\ref{proof_proposition_application_3_6}) and (\ref{proof_proposition_application_3_5}) with $ w^{1/2} = w_2^{s} $, $ \tilde{w}^{1/2} = w_1^s $, $ \beta_1 = \alpha_1 s  $, $ \beta_2 = \alpha_2 s $, $ p = \sigma/s $ and $ \beta = \frac{n-1}{2} + s $, we find
\begin{align*}
\big| \langle w_2^{{s+it}}(-\Delta-z)^{-s-it}w_1^{{s+it}} \varphi, \psi \rangle \big|
&\le Ce^{Ct^2}
||w_1^s||_{M^{\frac{2n}{\beta_1},2p}}
||w_2^s||_{M^{\frac{2n}{\beta_2},2p}}\\
&=Ce^{Ct^2}
||w_1||_{M^{\frac{2sn}{\beta_1},2sp}}^s
||w_2||_{M^{\frac{2sn}{\beta_1},2sp}}^s.
\end{align*}
In other words, (\ref{proof_proposition_application_3_1_preuve}) holds with $ s = \frac{n-1}{\alpha_1 + \alpha_2 -2}  $ which   belongs to  $ [ (n-1)/2 , (n+1)/2) ) $ and $ [1,\sigma) $ under our assumptions (note that, assuming $\sigma$ strictly greater than $ \frac{n-1}{\alpha_1 + \alpha_2 - 2} $ ensures that $ p = \sigma/s > 1 $). The result follows by interpolation (note that if $ n = 3 $ and $ \alpha_1 = \alpha_2  = 2$, one obtains directly (\ref{proposition_application_3_1}) from (\ref{proof_proposition_application_3_1_preuve}) with $s=1 \in [(n-1)/2 , (n+1)/2)$). 
\end{proof}

\section{Proofs of the main results}
\label{section_proof}
\setcounter{equation}{0}
In the present section, we show how to use the following Theorem \ref{theorem_resolvent_2main} and abstract techniques prepared in Section \ref{sectionabstraite} to prove all results stated in Section \ref{sectionresults}, except  Theorem \ref{theorem_3} and Corollary \ref{Corollaire-Mourre} which will be proved in Section \ref{sectionweakMourre}.


\begin{theorem}
\label{theorem_resolvent_2main}
{\rm (1)} Let $n\ge3$ and suppose that $V$ satisfies Assumption \ref{assumption_1} and that
\begin{align}
\label{theorem_resolvent_2main_1}
\big|\big||V|^\frac12g\big|\big|_{L^2}+\big|\big||x\cdot\nabla V|^\frac12g\big|\big|_{L^2}+\big|\big||x|Vg\big|\big|_{L^2}
\le C||g||_{\mathcal H^1},\quad g\in {\mathcal H}^1. 
\end{align}
Then there exists $C>0$ such that
$$
 |||x|^{-1}(H-z)^{-1}|x|^{-1} f||_{L^2} \leq C || f ||_{L^2}, \quad z\in \C\setminus[0,\infty) , \ f \in C_0^{\infty}({\mathbb R}^n \setminus 0). 
$$
{\rm (2)} Let $n=2$ and $V \in {\mathcal X}_2$ satisfy Assumption \ref{assumption_2}. Then
$$
||V^\frac12 (H-z)^{-1}V^\frac12 f||_{L^2} \leq C || f ||_{L^2}, \quad z \in \C\setminus[0,\infty) , \ f \in C_0^{\infty}({\mathbb R}^2 \setminus 0) .
$$

\end{theorem}

Note that \eqref{theorem_resolvent_2main_1} holds if $ V \in {\mathcal X}_n^{\sigma} $ with $ \frac{n-1}{2}<\sigma \leq \frac{n}{2}$. The proof of this theorem itself is based on the techniques of \cite{BVZ} which we follow closely.  However, we cannot use directly the result of  \cite{BVZ} since our assumptions are slightly different from theirs (see Remark \ref{remark_C}) so we give a complete proof in Appendix \ref{appendix_C}. 

\subsection{Proof of Theorem \ref{theorem_0}} \label{premiereapplication} 
If $n=2$, the statement is exactly  Theorem \ref{theorem_resolvent_2main} (2) so we assume  that $ n \geq 3 $.  We use the decomposition $ H = H_0 + Y^*Z $ of Section \ref{sectionabstraite} so we  let  $ V = Y^* Z $ with $ Y:= |x|V$, $ Z := |x|^{-1}$, and $ H_0 = - \Delta $. Recall that $|x|V\in M^{n,2 \sigma} $ by assumption. We may assume $w_1,w_2\ge0$ without loss of generality since if we write $w_j=\mathrm{sgn}\, w_j|w_j|$ then $\mathrm{sgn}\, w_j$ is bounded on $L^2$. 
 
Let us first prove the result with additional conditions that $w_1^{-1},w_2^{-1}\in  L^2_{\mathrm{loc}}$ and $w_1,w_2>0$. We shall use Proposition \ref{proposition_abstract_1} with $\mathcal H=L^2$, $\mathcal A=w_2L^2$, $\mathcal B=w_1L^2$ with norms $||\psi||_{\mathcal A}=|| w_2^{-1}\psi||_{L^2}$ and $||\varphi||_{\mathcal B}=|| w_1^{-1}\varphi||_{L^2}$, where we note that $\mathcal A,\mathcal B$ are Banach spaces under above additional conditions. By Proposition \ref{proposition_application_3} with $\sigma_j=\sigma$ and $\alpha_j=2$,  (\ref{proposition_abstract_1_1}) to (\ref{proposition_abstract_1_4}) are satisfied with
\begin{align*}
r_1&\le C||w_1||_{M^{n,2\sigma}}||w_2||_{M^{n,2\sigma}},\\
r_2,r_3&\le C||w_2||_{M^{n,2\sigma}}||{|x|V}||_{M^{n,2\sigma}}\le C||w_2||_{M^{n,2\sigma}},\\
r_4&\le C||w_1||_{M^{n,2\sigma}}|||x|^{-1}||_{M^{n,2\sigma}}\le C||w_1||_{M^{n,2\sigma}},
\end{align*}
where $C>0$ is independent of $w_1,w_2$ and $z\in \C\setminus[0,\infty)$. 
The condition (\ref{proposition_abstract_1_5}) with some $r_5$ (independent of $w_j$ and $z$) follows from Theorem \ref{theorem_resolvent_2main} (1). Hence we learn by \eqref{inegaliteresolvantefaible} that 
 $$
 |\langle(H-z)^{-1}\psi,\varphi\rangle|\le C||w_1||_{M^{n,2\sigma}}||w_2||_{M^{n,2\sigma}}||w_2^{-1}\psi||_{L^2}||w_1^{-1}\varphi||_{L^2},\quad z\in \C\setminus[0,\infty),
$$
for all $\psi\in L^2\cap w_2L^2$ and $\varphi\in L^2\cap w_1L^2$, which implies
$$
 |\langle w_1(H-z)^{-1}w_2\psi,\varphi\rangle|\le ||w_1||_{M^{n,2\sigma}}||w_2||_{M^{n,2\sigma}}||f||_{L^2}||g||_{L^2},\quad z\in \C\setminus[0,\infty),
$$
for all $f,g\in C_0^\infty(\R^n)$, where we have used \eqref{Fefferman} to see that $C_0^\infty(\R^n)\subset L^2\cap w_j^{-1}L^2$. By density and duality arguments, $w_1(H-z)^{-1}w_2$ extends a bounded operator on $L^2$ and satisfies 
$$
||w_1(H-z)^{-1}w_2f||_{L^2}\le C||w_1||_{M^{n,2\sigma}}||w_2||_{M^{n,2\sigma}}||{f}||_{L^2},\quad f\in C_0^\infty(\R^n),\ z\in \C\setminus[0,\infty). 
$$

For general $w_j\in M^{n,2\sigma}$, we set $w_j(\epsilon)=w_j+\epsilon \langle x\rangle ^{-2}$ and apply the above result to obtain
\begin{align}
\label{proof_theorem_0_1}
||w_1(\epsilon)(H-z)^{-1}w_2(\epsilon)f||_{L^2}\le C||w_1(\epsilon)||_{M^{n,2\sigma}}||w_2(\epsilon)||_{M^{n,2\sigma}}||{f}||_{L^2}. 
\end{align} 
It is not hard to see that $w_1(\epsilon)(H-z)^{-1}w_2(\epsilon)f\to w_1(H-z)^{-1}w_2f$ for any $f\in C_0^\infty(\R^n)$ and $||w_j(\epsilon)||_{M^{n,2\sigma}}\to ||w_j||_{M^{n,2\sigma}}$ as $\epsilon\to0$ (note that $||w_j||_{M^{n,2\sigma}}\le ||w_j(\epsilon)||_{M^{n,2\sigma}}\le ||w_j||_{M^{n,2\sigma}}+C\epsilon$). 
Hence, letting $\epsilon\to0$ in \eqref{proof_theorem_0_1}, we have the desired bound for $w_1(H-z)^{-1}w_2$. 
 \hfill $ \Box $


\subsection{Proof of Corollary \ref{corollary_theorem_0}}
It follows from  Theorem \ref{theorem_0} with (\ref{H-smooth_2}) and (\ref{H_supersmooth_2}).  \hfill $ \Box $


\subsection{Proof of Theorem  \ref{theoremevariante} }
We wish to use  Proposition \ref{proposition_abstract_1} with $ \mathcal A={\mathcal B} = L^{2_*,2} $, $H_0=-\Delta$, $H=-\Delta+Y^*Z$, $ Y:= |V|^{\frac{1}{2}} $ and $ Z := \mbox{sgn}(V) |V|^{\frac{1}{2}}$. Using the condition that $ |V|^{\frac{1}{2}} $ belongs to $ L^{n,\infty} $,  H\"older's inequality (\ref{HolderLorentz}) yields
\begin{align}
\label{proof_theoremevariante}
|| Y g ||_{L^2} + ||Z g ||_{L^2} \leq C ||g ||_{L^{2^*,2}},\quad 
|| Yf ||_{L^{2_*,2}}\le C||f||_{L^2},
\end{align}
so using Proposition \ref{lemma_resolvent_1}, we obtain
$$ || Y R_0 (z) f ||_{L^{2}} +  || Z R_0 (z) f ||_{L^2} \leq C || f ||_{L^{2_*,2}} , $$
for all $ z \in \C \setminus [0,\infty) $ and $ f \in L^2 \cap L^{2_*,2} $ i.e. the conditions (\ref{proposition_abstract_1_1})--(\ref{proposition_abstract_1_4})
are satisfied (uniformly in $z$). The bound (\ref{proposition_abstract_1_5}) follows from  Theorem \ref{theorem_resolvent_2main} (1). Then we obtain 
$$
 |\langle(H-z)\psi,\varphi\rangle|\le C||\psi||_{L^{2_*,2}}||\varphi||_{L^{2_*,2}},\quad \psi,\varphi\in L^{2_*,2}\cap L^2,
 $$
which, together with duality argument (see paragraph \ref{sectionnotationLorentz}), implies the assertion.  \hfill $ \Box $

\subsection{Proof of Theorem \ref{theorem_1}}
Let $H_0,H,Y$ and $Z$ be as in the proof of Theorem  \ref{theoremevariante}. 
Recall that the solution to \eqref{equation_1} is given by $u=e^{-itH}\psi-i\Gamma_HF$. First of all, it was proved by \cite[Theorem 10.1]{KT} that $e^{it\Delta}$ and $\Gamma_{-\Delta}$ satisfy
\begin{align}
\label{proof_theorem_1_1}
||e^{it\Delta}\psi||_{L^p_tL^{q,2}_x}\le C||\psi||_{L^2_x},\quad
||\Gamma_{-\Delta}F||_{L^p_tL^{q,2}_x}\le C||F||_{L^{\tilde p'}_tL^{\tilde q',2}_x}
\end{align}
for any admissible pairs $(p,q)$ and $(\tilde p,\tilde q)$. Also recall that, for any $1\le p,q<\infty$ and any dense subset $\mathcal D\subset L^{q,2}$, simple functions $G:[-T,T]\to \mathcal D$ are dense in $L^p_tL^{q,2}_x$. 

Consider the non-endpoint estimates for $n\ge3$. We shall use Theorem \ref{theorem_inhomogeneous_2} with $\mathcal B:=L^{q',2}$ and $b\equiv1$. (S1$'$) is exactly the first estimate in \eqref{proof_theorem_1_1}. Since $b\equiv1$, (S2$'$) and (S3$'$) follow from Proposition \ref{proposition_application_3} or Theorem \ref{theorem_0} (1), respectively. Theorem \ref{theorem_inhomogeneous_2} thus implies homogeneous estimates:
\begin{align*}
|| e^{-itH} \psi ||_{L^p_t L^{q,2}_x} \le C || \psi ||_{L^2_x}
\end{align*}
for all non-endpoint admissible pair $(p,q)$. Then,  a standard argument using the Christ-Kiselev Lemma and the duality  (see, {\it e.g.}, \cite[Lemma 7.4]{BoTz}) implies the inhomogeneous estimates:
$$
||\Gamma_H F ||_{L^p_t L^{q,2}_x} \le C || F||_{L^{\tilde p'}_t L^{{\tilde q',2}}_x}
$$
for all non-endpoint admissible pairs $(p,q)$ and $(\tilde p ,\tilde  q)$. 

In the case of the endpoint estimate for $n\ge3$ under the additional condition $V\in L^{\frac n2,\infty}$, we shall use Theorem \ref{theorem_inhomogeneous_1} with $\mathcal A={\mathcal B} := L^{2_*,2} $. (S1) follows from \eqref{proof_theorem_1_1}. To derive the condition (S2), we observe that the second estimate in \eqref{proof_theorem_1_1} and its dual estimate, together with \eqref{proof_theoremevariante}, imply that 
\begin{align}
\label{proof_theorem_1_3}
||Y\Gamma_{-\Delta}F||_{L^2_TL^2_x}
+||Y\Gamma_{-\Delta}^*F||_{L^2_TL^2_x}
+||Z\Gamma_{-\Delta}F||_{L^2_TL^2_x}
\le C||F||_{L^2_TL^{2_*,2}_x}.
\end{align}
This estimate, together with Lemma \ref{sansHille}, yield that $Y\Gamma_{-\Delta}F(t)=\overline{Y\Gamma_{-\Delta}}F(t)$, $Y\Gamma_{-\Delta}^*F(t)=\overline{Y\Gamma_{-\Delta}^*}F(t)$ and $Z\Gamma_{-\Delta}F(t)=\overline{Z\Gamma_{-\Delta}}F(t)$ for any simple function $F:[-T,T]\to L^2\cap L^{2_*,2}$ and for a.e. $t\in\R$. In particular, the condition (S2) follows form \eqref{proof_theorem_1_3}. The condition (S3) follows from Theorem \ref{theorem_0} (1) with $w_1=w_2=|V|^{1/2}$ and Corollary \ref{corollary_supersmooth}. The condition (S4) is trivial since $ b \equiv 1 $. Therefore, Theorem \ref{theorem_inhomogeneous_1} together with density and duality arguments implies
$$
|| e^{-itH} \psi ||_{L^2_t L^{2^*,2}_x} \le C || \psi ||_{L^2_x}, \quad || \Gamma_H F ||_{L^2_t L^{2^*,2}_x} \le C || F||_{L^2_t L^{2_*,2}_x}.
$$


When $n=2$, we use the same decomposition for $V$ and Theorem \ref{theorem_inhomogeneous_2} with $ {\mathcal B} = L^{q',2} $. The condition (S1$'$) again follows from \eqref{proof_theorem_1_1}, while the $\Delta$- (resp. $H$-) smoothness of $Y$ (resp. $Z$)  follows from Proposition \ref{lemma_resolvent_3} (resp. Theorem \ref{theorem_0} (2)). Hence Theorem \ref{theorem_inhomogeneous_2} implies the homogeneous estimates. Inhomogeneous estimates are again derived by using the Christ-Kiselev Lemma.  
\hfill $ \Box $

\subsection{Proof of Corollary \ref{corollary_2}}
\label{proof_corollary_2}
Let us set $H_0=-\Delta+V_1$, $H=H_0+V_2$, $Y=\mbox{sgn}(V)|V_2|^{\frac 12}$ and $Z=|V_2|^\frac12$. As in paragraph \ref{realizations}, \eqref{assumption_1_1} and \eqref{Fefferman} imply that both $H$ and $H_0$ are proportional to $-\Delta$ in the sense of forms on $C_0^\infty(\R^n)$ provided $||V||_{M^{\frac n2,\sigma}}$ is small enough. In particular, we see that 
$$
D(H)\cup D(H_0)\subset \mathcal H^1=D(Q_{H_0})=D(Q_H)\subset D(Z)=D(Y),
$$
which, together with the density of $D(Y)=D(Z)$ in $L^2$, implies that $Y$ and $Z$ are relatively  bounded with respect to both $H_0$ and $H$. 

Let us first show the $H$-supersmoothness of $Z$ when $||V||_{M^{\frac n2,\sigma}}$ is sufficiently small. The resolvent identity \eqref{resolventeidentity1} with $(u,v)=(Zf,Zg)$ for $f,g\in D(Z)$ (note that $Z$ is self-adjoint) implies
$$
\langle R_H(z)Zf,Zg\rangle =\langle R_{H_0}(z)Zf,Zg\rangle -\langle ZR_H(z)Zf,YR_{H_0}(\overline z)Zg\rangle.
$$
Theorem \ref{theoremevariante} with $w_1=w_2=|V_2|^\frac12$ for $R_{H_0}(z)$ then shows, for $z\in \C\setminus[0,\infty)$,
$$
||ZR_H(z)Zf||_{L^2}=\sup_{||g||=1}|\langle R_H(z)Zf,Zg\rangle|\le C||V_2||_{M^{\frac n2,\sigma}}(||f||_{L^2}+||ZR_H(z)Zf||_{L^2})
$$
with $C>0$ independent of $V_2$ and $z$. Therefore, taking $||V_2||_{M^{\frac n2,\sigma}}$ small enough, one has 
$$
||ZR(z)Zf||_{L^2}\le C||V_2||_{M^{\frac n2,\sigma}}(1-C||V_2||_{M^{\frac n2,\sigma}})^{-1}||f||_{L^2}
$$
for $f\in D(Z)$ which implies the $H$-supersmoothness since $D(Z) (\supset \mathcal H^1)$ is dense in $L^2$. 

Next we prove the assertion in the non-endpoint case. We use Theorem \ref{theorem_inhomogeneous_2} with $\mathcal B=L^{q',2}$. The conditions (S1$'$) and (S2$'$) follow from Theorem \ref{theorem_1} (1) and Theorem \ref{theorem_0} (1), respectively, while (S3$'$) is an immediate consequence of the $H$-supersmoothness of $Z$. Therefore, Theorem \ref{theorem_inhomogeneous_1} can be applied to obtain non-endpoint Strichartz estimates for $e^{-itH}$. Estimates for the Duhamel operator $\Gamma_H$ again follow from the estimates for $e^{-itH}$ and the Christ-Kiselev lemma. 

In order to derive the assertion for the endpoint case under the smallness of $||V_2||_{L^{\frac n2,\infty}}$, we shall use Theorem \ref{theorem_inhomogeneous_1} with $\mathcal A=\mathcal B=L^{2_*,2}$. The conditions (S1) and (S2) again follow from Theorem \ref{theorem_1} (1) and \eqref{proof_theoremevariante} with $H$ replaced by $H_0$, where we have used the  condition $V_2\in L^{\frac n2,\infty}$ to obtain \eqref{proof_theoremevariante}. (S3) is exactly the $H$-supersmoothness of $Z$, which follows from the same argument as above since $||V_2||_{M^{\frac n2,\sigma}}\le C||V_2||_{L^{\frac n2,\infty}}$. Finally, (S4) is trivial since $b\equiv1$. Thus Theorem \ref{theorem_inhomogeneous_1} gives us the assertion in the endpoint case.
\hfill $ \Box $

\section{A weakly conjugate operator method} \label{sectionweakMourre}
\setcounter{equation}{0}

In this section, we consider operators $ H = - \Delta + V $  with $ V $ satisfying Assumption \ref{assumptionalternative}.  Let us recall the definition of the usual group of dilations 
\begin{align}
\label{dilation}
e^{itA}f(x)=e^{tn/2}f(e^{t}x),
\end{align}
which is the strongly continuous unitary group on $ L^2 (\R^n) $ with generator
\begin{align}
\label{conjugate_operator}
A:=\frac{1}{2i}(x\cdot\nabla+\nabla\cdot x). 
\end{align} 


In Assumption \ref{assumptionalternative}, the condition (1) 
  with $ \ell = 1,2 $ allows to define the commutators
\begin{align*}
[H,iA]&=-2\Delta-x \cdot \nabla V=2H-2V- x \cdot \nabla V,\\
[[H,iA],iA]&=-4\Delta+( x \cdot \nabla )^2V=2[H,iA]+2 x \cdot \nabla  V+( x \cdot \nabla)^2V
\end{align*}
as  sesquilinear forms on $ C_0^{\infty} (\R^n \setminus 0) \times C_0^{\infty} (\R^n \setminus 0) $, i.e.
\begin{align}
 Q_{[H,iA]}(f,g) &:= Q_H (f,iAg) - Q_H (Af, ig) \nonumber \\ 
 &=  Q_H (f,g) -  \langle (2V + x \cdot \nabla V)  f , g \rangle  \label{controleparS},\\
 Q_{[[H,iA],iA]}(f,g) &:=  Q_H (f,iAg) - Q_H (Af, ig)  \nonumber  \\
 &= 2 Q_{[H,iA]}(f,g) +   \langle (2 x \cdot \nabla  V + (x \cdot \nabla V )^2 V)  f , g \rangle . \label{controleparSbis}
\end{align}
The first condition in Assumption \ref{assumptionalternative} (4) implies that 
\begin{align}
\label{proof_lemma_resolvent_5_0}
|Q_{[H,iA]}(f,g)|\le C ||f||_{\mathcal G^1} ||g||_{\mathcal G^1},\quad f,g\in C_0^{\infty}(\R^n \setminus 0),
\end{align}
showing that the sesquilinear form $  Q_{[H,iA]}(f,g )  $ extends from $C_0^{\infty} (\R^n \setminus 0) \times C_0^{\infty} (\R^n \setminus 0)$  to a continuous sesquilinear form on ${\mathcal G}^1\times {\mathcal G}^1$, still denoted by  $Q_{[H,iA]}(f,g)$ (see paragraph \ref{realizations} for $ {\mathcal G}^1 $).  The condition (3) in Assumption \ref{assumptionalternative} implies in particular that
$$ Q_S (f) := Q_{[H,iA]} (f,f) \geq 0 . $$
This allows to define  $ D (S^{1/2}) \subset L^2 $ as the closure of $ C_0^{\infty} (\R^n \setminus 0) $ for the norm $ \big( ||f||_{L^2}^2 + Q_S (f) \big)^{1/2} $. By (\ref{proof_lemma_resolvent_5_0}), $ {\mathcal G}^1 $ is continuously and densely embedded into $ D (S^{1/2}) $. The sesquilinear form $ Q_{[H,iA]} $ then extends continuously to $ D (S^{1/2}) $ and  gives rise to a nonnegative self-adjoint operator $ S : D (S) \rightarrow L^2 $, such that
$$ Q_{[H,iA]} (f,g) = \langle f , S g \rangle, \quad f \in D (S^{1/2}), \ g \in D (S) .$$
Note that the notation $ D (S^{1/2}) $ is unambiguous since this space is exactly the domain of $ \sqrt{S} $ defined by functional calculus of non-negative self-adjoint operators.

The second condition in Assumption \ref{assumptionalternative} (4) (see also (\ref{controleparS}) and (\ref{controleparSbis})) ensures that 
\begin{align}
\label{proof_lemma_resolvent_6_6}
| Q_{ [[H,iA],iA]} (f,f ) \ |\le C Q_{[H,iA]} (f,f)  ,\quad f\in C_0^{\infty} (\R^n \setminus 0),
\end{align}
so, by (\ref{proof_lemma_resolvent_5_0}), the form $ Q_{[[H,iA],iA]}  $ can be extended continuously to $ {\mathcal G}^1 $ on which it still satisfies  (\ref{proof_lemma_resolvent_6_6}).
The estimate (\ref{proof_lemma_resolvent_6_6}) is  technically important  in the proof of the following theorem.

\begin{theorem} \label{theorem_resolvent_7} If $ \kappa > 0 $ is large enough, then $ S^{1/2} (A + i \kappa)^{-1} $ is $H$-supersmooth, that is
$$
 \sup_{z \in \C \setminus \R} \left|\left\langle (A-i\kappa)^{-1} S^{1/2} f , (H-z)^{-1}  (A-i\kappa)^{-1} S^{1/2} g  \right\rangle \right| \leq C ||f||_{L^2} ||g ||_{L^2} ,\quad
f , g \in {\mathcal G}^{1}. 
$$
\end{theorem}

This theorem can be seen as a consequence of some version of the weakly conjugate operator method (see \cite{BKM,BoMu,Richard}), in that it only uses the non-negativity of $ [H,iA] $ and the upper bound (\ref{proof_lemma_resolvent_6_6}). Our version is  fairly simpler than in the previous references  for we do not use interpolation spaces, nor even use that $ || S^{1/2} u ||_{L^2} $ defines a norm. The stronger lower bound (3) in Assumption \ref{assumptionalternative}  is only used to obtain Theorem \ref{theorem_3}, i.e. to replace the operator $ (A - i \kappa)^{-1} S^{1/2} $ by the physical weights $ |x|^{-1} $ or  $w$ (when $n=2$).

Before proving Theorem \ref{theorem_resolvent_7}, we show how it implies Theorem \ref{theorem_3}.

\begin{proof}[Proof of Theorem \ref{theorem_3}]
We prove the cases $ n \geq 3 $ and $ n =2 $ simultaneously by setting $ w (x) = |x|^{-1} $ if $ n \geq 3 $; indeed, in Assumption \ref{assumptionalternative} (3), the lower bound in dimension $ n \geq 3$ can be replaced by $ ||\nabla f||_{L^2}^2 + \big| \big|  |x|^{-1} f \big| \big|_{L^2}^2 $ thanks to the Hardy inequality (and up to possibly changing $ \delta_0$). The result will then be clearly a consequence of
\begin{align}
\big| \langle (H-z)^{-1}  w  \varphi , w \psi \rangle \big| \leq C ||\varphi||_{L^2} ||\psi||_{L^2} , \quad z \in {\mathbb C} \setminus {\mathbb R}, \ \ \varphi, \psi \in C_0^{\infty} ({\mathbb R}^n \setminus 0). \label{borneaprouverMourre}
\end{align}
Write first
$$ w \varphi = \lim_{\epsilon \downarrow 0} (A - i \kappa)^{-1} S^{\frac{1}{2}} (S + \epsilon)^{-\frac{1}{2}} (A - i \kappa) w \varphi $$
using that $  (A - i \kappa) w \varphi \in L^2 $ (since $ w \in C^1 ({\mathbb R}^n \setminus 0) $) and that $S$ is a nonnegative self-adjoint operator with no $0$ eigenvalue by Assumption \ref{assumptionalternative} (3). Thanks to Theorem \ref{theorem_resolvent_7}, we have
\begin{align}
 \big| \langle (H-z)^{-1}  w  \varphi , w \psi \rangle \big| \leq C \sup_{\epsilon > 0} ||  (S + \epsilon)^{-\frac{1}{2}} (A - i \kappa) w \varphi||_{L^2} || (S + \epsilon)^{-\frac{1}{2}} (A - i \kappa) w \psi||_{L^2} \label{poidsrigoureux}
\end{align}
where the constant $ C $ is independent of $ z $ and $ \varphi, \psi $.  Note here that  $ (S + \epsilon)^{-\frac{1}{2}} (A - i \kappa) w \varphi  $ does not clearly belong to $  {\mathcal G}^1$ (and likewise with $ \psi $), as is required in Theorem \ref{theorem_resolvent_7}; however for fixed $ \epsilon $, it can be approached by a sequence of $ {\mathcal G}^1 $ which allows to fully justify (\ref{poidsrigoureux}). Then, by writing
$$ (A - i \kappa) w = \left( \frac{\nabla \cdot x}{i} - \frac{n}{2i} - i \kappa \right) w  $$
and using on the other hand that Assumption \ref{assumptionalternative} (3) implies
$$ \big| \big| (S + \epsilon)^{-\frac{1}{2}} \nabla u \big| \big|_{L^2} +\big| \big| (S + \epsilon)^{-\frac{1}{2}} w u \big| \big|_{L^2}  \leq C || u ||_{L^2}, \quad u \in C_0^{\infty} ({\mathbb R}^n \setminus 0) , $$
with $  C $ independent of $ \epsilon $, we see that the right hand side of (\ref{poidsrigoureux}) is bounded by
$$ C \big( ||x w \varphi||_{L^2} +  || \varphi ||_{L^2}  \big) \big(  ||x w \psi ||_{L^2} + || \psi ||_{L^2}  \big) .$$ Since $ |x|w $ is bounded by assumption, this yields (\ref{borneaprouverMourre}). \end{proof}

\begin{proof}[Proof of Corollary \ref{Corollaire-Mourre}]
It follows from Theorem \ref{theorem_3} together with (\ref{H-smooth_2}) and (\ref{H_supersmooth_2}).
\end{proof}

The rest of the section is devoted to the proof of Theorem \ref{theorem_resolvent_7}.  We let $ {\mathcal G}^{-1} $ be the (anti)dual of $ {\mathcal G}^1 $, i.e. the space of continuous conjugate linear  forms on $ {\mathcal G}^1 $. To avoid any ambiguity, we denote by $ \langle u , f \rangle_{{\mathcal G}^{-1},{\mathcal G}^1}  := u (f)$ the duality between $ u \in {\mathcal G}^{-1} $ and $ f \in {\mathcal G}^1 $ (it is linear in $u$ and conjugate linear in $ f $). We keep the notation $ \langle \cdot , \cdot \rangle $ for the inner product on $L^2$ only.

   Then, we define three linear continuous operators $ \tilde{H} , \tilde{S} , \tilde{S}^{\prime} : {\mathcal G}^1 \rightarrow {\mathcal G}^{-1} $ by
\begin{align}
 \tilde{H} f := Q_H (f,\cdot), \quad \tilde{S} f := Q_{[H,iA]}(f,\cdot), \quad \tilde{S}^{\prime} = Q_{[[H,iA],iA]} (f,\cdot) \quad f \in {\mathcal G}^1 .  \label{introductionSprime}
\end{align}
The operators $ \tilde{H} $ and $ \tilde{S} $ are  respectively  extensions of $ H $ and $ S $ to $ {\mathcal G}^1 $, in the sense that
$$ \tilde{H} f = H f \quad \mbox{if} \ f \in  D (H), \quad \tilde{S} f = S f \quad \mbox{if} \ f \in {\mathcal G}^1 \cap D (S) $$
or, to be completely rigorous, $ \tilde{H}f = \langle Hf , \cdot \rangle $ and $ \tilde{S}f = \langle S f , \cdot \rangle $ respectively. 
\begin{proposition} \label{propinversibilite} Let $ z \in \C \setminus \R $. Let $ \epsilon \in \R $ such that $ \epsilon {\emph Im}(z) \geq 0 $ (i.e. either $ \epsilon = 0 $ or $ \epsilon $ and $  {\emph Im}(z) $ have the same strict sign). Then
$$ \tilde{H} - z - i \epsilon \tilde{S} : {\mathcal G}^1 \rightarrow {\mathcal G}^{-1} $$
  is an isomorphism.  The multiplication by $z$ means $ f \mapsto \langle z f , . \rangle = z \langle f , . \rangle $.
\end{proposition}

\begin{proof}
Let us assume e.g. that $ \mbox{Im}(z)> 0 $ and $ \epsilon \geq 0 $. Then, for all $ f \in {\mathcal G}^1 $, one has
\begin{align}
  \mbox{Re} \langle (\tilde{H} - z - i \epsilon \tilde{S}) f , f \rangle_{{\mathcal G}^{-1},{\mathcal G}^1} &= Q_H (f,f) - \mbox{Re}(z) ||f||_{L^2}^2  \nonumber \\
 -  \mbox{Im} \langle (\tilde{H} - z - i \epsilon \tilde{S}) f , f \rangle_{{\mathcal G}^{-1},{\mathcal G}^1} &=  \mbox{Im}(z)||f||_{L^2}^2 + \epsilon Q_{[H,iA]}(f,f) \ \geq \ \mbox{Im}(z)||f||_{L^2}^2 . \label{partieimaginaire}
\end{align}
 Plugging the estimate of the second line in the first line implies easily the coercivity estimate
$$ ||f||_{{\mathcal G}^1}^2 = Q_H (f,f) + ||f||_{L^2}^2 \leq \left( 1 + \frac{ |{\rm Re}(z)|+1}{\mbox{Im}(z)}\right) \left| \langle (\tilde{H} - z - i \epsilon \tilde{S}) f , f \rangle_{{\mathcal G}^{-1},{\mathcal G}^1} \right|  . $$
 One then has the expected bijectivity by an application of the Lax-Milgram Theorem to the sesquilinear form $ (f,g) \mapsto \langle (\tilde{H} - z - i \epsilon \tilde{S}) f , g \rangle_{{\mathcal G}^{-1},{\mathcal G}^1} = Q_H (f,g) - i \epsilon Q_{[H,iA]}(f,g) - z \langle f , g \rangle $.
 \end{proof}

This proposition allows to consider $ G_{\epsilon} (z) := ( \tilde{H} - z - i \epsilon \tilde{S} )^{-1}  $. We record that, upon the identification of any $f \in L^2 $ with the form $ \scal{f,\cdot} $ which belongs to $ {\mathcal G}^{-1} $, one has for $ \epsilon = 0 $
$$ G_0 (z) f = (H-z)^{-1} f, \quad f \in L^2 , $$
which simply follows from the fact that $ Q_H \big((H-z)^{-1} f,g \big) = \langle H (H-z)^{-1} f  , g \rangle $ for any $ g \in {\mathcal G}^1 $. Also, it is useful and not hard to check that $ G_{\epsilon} (z) $ and $ G_{-\epsilon} (\bar{z}) $ are adjoint to each other in the precise sense that,   for all $u,v \in {\mathcal G}^{-1} $,
\begin{align}
  \langle  u ,  \ G_{\epsilon} (z) v \rangle_{{\mathcal G}^{-1}, {\mathcal G}^1} = \overline{  \langle  v ,  \ G_{-\epsilon} (\bar{z}) u \rangle_{{\mathcal G}^{-1}, {\mathcal G}^1}  } . \label{definitionpreciseadjoint}
\end{align}  
To derive (\ref{definitionpreciseadjoint}), it suffices to write $ u  = \big( \tilde{H} - \bar{z} + i \epsilon \tilde{S} \big) G_{-\epsilon} (\bar{z}) u $ and  to use the symmetry of  $ Q_H $ and $ Q_{[H,iA]} $. 
We also record at this stage the useful formula 
 \begin{align}
 \frac{d}{d \epsilon} G_{\epsilon} (z) = G_{\epsilon}(z) i \tilde{S} G_{\epsilon}(z)   \label{formulederiveeG}
\end{align}
which follows from the differentiability of $ \epsilon \mapsto \tilde{H} - z - i \epsilon \tilde{S} $    in operator norm.

 \begin{proposition} \label{bornedissipative} Let $ B : {\mathcal G}^1 \rightarrow L^2 $ be a bounded linear map. Then, for $ \epsilon {\emph Im} (z) > 0 $,
 $$ \big| \big| S^{1/2} G_{\epsilon} (z) B^*  \big| \big|_{L^2 \rightarrow L^2} \leq |\epsilon|^{- \frac{1}{2}} \big| \big| B G_{\epsilon} (z) B^*  \big| \big|_{L^2 \rightarrow L^2 }^{\frac{1}{2}} .  $$
 \end{proposition}   
 
\begin{proof}
Let $ f \in L^2 $. Denoting for simplicity $ B^* f $ instead of $ B^* ( \langle f , \cdot \rangle ) $, one has
 \begin{align*}
   \big| \big| S^{1/2} G_{\epsilon} (z) B^* f \big| \big|_{L^2 }^2 &= Q_{[H,iA]} (G_{\epsilon}(z)B^* f , G_{\epsilon}(z) B^* f) \\ 
   &\le \frac{1}{\epsilon} \Big(  \epsilon Q_{[H,iA]} (G_{\epsilon}(z)  B^* f , G_{\epsilon}(z) B^* f )  + \mbox{Im}(z) || G_{\epsilon}(z) B^* f ||^2 \Big) ,
 \end{align*}
 where, according to (\ref{partieimaginaire}), the parentheses in the second line is 
 $$ - \mbox{Im} \langle (\tilde{H} - z - i \epsilon \tilde{S}) G_{\epsilon}(z) B^* f , G_{\epsilon}(z) B^* f \rangle_{{\mathcal G}^{-1},{\mathcal G}^1} = - \mbox{Im} \langle B^* f , G_{\epsilon}(z) B^* f \rangle_{{\mathcal G}^{-1},{\mathcal G}^1} . $$
 Thus,  $\big| \big| S^{1/2} G_{\epsilon} (z) B^* f \big| \big|_{L^2 }^2 \leq |\epsilon|^{-1} |\langle f , B G_{\epsilon} (z) B^* f \rangle | $ yields the result. 
\end{proof}

 

To prepare all the material needed to follow the usual differential inequality technique of Mourre, we need a technical result.

\begin{proposition}
\label{lemma_resolvent_6}
{\rm (1)} For all $t\in\R$, $e^{itA}$ leaves $\mathcal G^1$ invariant and there exists $c_0\ge0$ such that 
$$ ||e^{itA}f ||_{{\mathcal G}^1}\le e^{c_0|t|} ||f ||_{{\mathcal G}^1}, $$ for all $ t\in\R $ and $ f \in {\mathcal G}^1 $.
In particular, for $ | \kappa | > c_0 $, $ (A+i\kappa)^{-1} $ maps $ {\mathcal G}^1 $ into itself continuously.\\
{\rm (2)} There exists $c_1\ge0$ such that, for all $f \in {\mathcal G}^1 $ and $ t \in \R $,
$$ \big| \big| S^{1/2} e^{itA} f \big| \big|_{L^2}\le e^{c_1|t|} || S^{1/2}f ||_{L^2} .$$
In particular, for $ |\kappa| > \max (c_0,c_1) $, there exists $ C_{\kappa} $ such that
\begin{align}
  || S^{1/2} (A+i\kappa)^{-1} f ||_{L^2} \leq C || S^{1/2} f ||_{L^2}, \quad f \in {\mathcal G}^1 .  \label{pourinitialisation}
\end{align}  
\end{proposition}

\begin{proof}
Let $f\in C_0^{\infty}(\R^n \setminus 0)$.  Since $ C_0^{\infty}(\R^n \setminus 0)$ is stable by $ e^{itA}$  by the explicit formula \eqref{dilation}, the quantity $Q(t):= Q_H (e^{itA}f,e^{itA}f)$ is well defined. We shall check its differentiability in $t$. To this end, it suffices to check the differentiability at $t=0$ by the group property of $e^{itA}$. We compute
\begin{equation}
\begin{aligned}
\label{proof_lemma_resolvent_6_1}
Q(t) - Q(0)=Q_H  (e^{itA}f-f,e^{itA}f-f )+ Q_H( f,e^{itA}f-f\ )+ Q_H ( e^{itA}f-f,f ),
\end{aligned}
\end{equation}
where the second and third terms of the right hand side satisfy
\begin{align}
\label{proof_lemma_resolvent_6_1_1}
\frac{d}{dt}\Big( Q_H( f,e^{itA}f-f )+ Q_H ( e^{itA}f-f,f )\Big)\Big|_{t=0}= Q_{[H,iA]}( f,f ), 
\end{align}
since $t \mapsto e^{itA} f $ is differentiable as a $ {\mathcal G}^1 $ valued map.
Next we shall show 
\begin{align}
\label{proof_lemma_resolvent_6_1_2}
Q_H(e^{itA}f-f,e^{itA}f-f )=O(t^2),\quad |t|\le1. 
\end{align}
To treat the gradient term in $ Q_H $, we use the representation $e^{itA}f-f=\int_0^tiAe^{isA}fds$ to see
\begin{equation}
\begin{aligned}
\label{proof_lemma_resolvent_6_2}
\big| \big| \nabla (e^{itA}f-f) \big| \big|_{L^2}^2
\le |t|^2\sup_{|t|\le1} \big| \big| \nabla Ae^{itA}f \big| \big|_{L^2}^2
\le C|t|^2 ||\nabla Af||_{L^2}^2,
\end{aligned}
\end{equation}
where in the last line we have used the formula $e^{-itA}\nabla e^{itA}=e^t\nabla$ and the fact $||e^{itA}||_{\mathbb B(L^2)}=1$. For the potential term, we consider two cases $n\ge3$ or $n=2$ separately. Suppose $n\ge3$ and $V\in L^{\frac n2,\infty}_{\rm loc}$.  Let $K \Subset \R^n$ contain $ {\rm supp} (e^{itA}f)$ for $ |t| \leq 1 $. By (\ref{HolderLorentz}) and (\ref{SobolevLorentz})
\begin{equation}
\begin{aligned}
\label{proof_lemma_resolvent_6_3}
|\langle V(e^{itA}f-f),e^{itA}f-f\rangle|
&\le C \big| \big| {\mathds 1}_K |V|^{\frac12} \big| \big|_{L^{n,\infty}}^2 \big| \big| \nabla (e^{itA}f-f) \big| \big|_{L^2}^2\\
&\le C|t|^2 \big| \big| {\mathds 1}_K |V|^{\frac12} \big| \big|_{L^{n,\infty}}^2 ||\nabla Af||_{L^2}^2 
\end{aligned}
\end{equation}
from which (\ref{proof_lemma_resolvent_6_1_2}) follows.
Next we let $n=2$ and decompose  $V=V_1+V_2$ with $V_1\in L^1_{\rm loc}$ and $r^2V_2\in L^\infty_{\rm loc}$. By H\"older's inequality, the first potential satisfies
\begin{align}
\label{proof_lemma_resolvent_6_4}
|\langle V_1(e^{itA}f - f),e^{itA}f - f\rangle |\le ||V||_{L^1(K)} ||e^{itA}f-f||_{L^\infty(K)}^2\le C ||V||_{L^1(K)}|t|^2 ||Af||_{L^\infty(K)}^2.
\end{align}
For the second potential, since $A=-ix\cdot\nabla-i$ and $e^{-itA}|x|^{-1}e^{itA}=e^{t}|x|^{-1}$, we have
$$
|x|^{-1} (e^{itA}f- f ) = i \int_0^t e^s e^{isA} \big(|x|^{-1}A f \big)ds
$$
so that $ \big| \big| |x|^{-1} (e^{itA}f- f ) \big| \big|_{L^2} \le C |t| \big( \big| \big| | x|^{-1} f \big| \big|_{L^2} + || \nabla f  ||_{L^2} \big) $ for $ |t| \leq 1 $ and hence
\begin{align}
\label{proof_lemma_resolvent_6_5}
|\langle V_2e^{itA}f,e^{itA}f\rangle |
\le C|t|^2 \big| \big| |x|^2 V_2 \big| \big|_{L^\infty(K)}(||\nabla f||_{L^2}^2+ \big| \big| |x|^{-1}f \big| \big|_{L^2}^2). 
\end{align}
Then \eqref{proof_lemma_resolvent_6_2} to \eqref{proof_lemma_resolvent_6_5}  show \eqref{proof_lemma_resolvent_6_1_2}. Moreover, by (\ref{proof_lemma_resolvent_6_1}), \eqref{proof_lemma_resolvent_6_1_1} and \eqref{proof_lemma_resolvent_6_1_2},    we have 
$$
\frac{d}{dt} Q_H(e^{itA}f,e^{itA}f )\Big|_{t=0}= Q_{[H,iA]}(f,f ).
$$
Using more generally the differentiability at any $t$, we obtain formula
\begin{align}
\label{expressionaderiverA}
Q_H ( e^{itA}f,e^{itA}f )
=Q_H(f,f)+\int_0^t Q_{[H,iA]} (e^{isA}f,e^{isA}f )ds. 
\end{align}
Combining this with  \eqref{proof_lemma_resolvent_5_0} and the fact $||e^{itA}f||_{L^2}=||f||_{L^2}$ implies
$$
||e^{itA}f||_{\mathcal G^1}^2\le ||f||_{\mathcal G^1}^2+\frac{c_0}{2}\int_0^{|t|} ||e^{isA}f||_{\mathcal G^1}^2ds, 
$$
with  $c_0 = 2C$ coming from (\ref{proof_lemma_resolvent_5_0}). Gronwall's inequality then shows $ ||e^{itA}f||_{\mathcal G^1}\le e^{c_0|t|} ||f||_{{\mathcal G}^1}$ for $f\in C_0^{\infty}(\R^n \setminus 0)$, which remains true on $ {\mathcal G}^1 $ by density.
The boundedness of $ (A + i \kappa)^{-1} $, say for $ \kappa > 0 $, follows from the fact that $ (A+i \kappa)^{-1} = i^{-1} \int_0^{+\infty} e^{-t \kappa} e^{itA} dt $. 

The proof of the second assertion is similar. Indeed,  $ x \cdot \nabla V$ satisfies the same conditions as $V$, namely $ x \cdot \nabla V\in L^{\frac n2,\infty}_{\rm loc}$ for $n\ge3$ or $x \cdot \nabla V\in L^1_{\rm loc}+|x|^{-2}L^\infty_{\rm loc}$ for $n=2$. This allows to differentiate $  Q_{S} (e^{itA}f) $ in $t$ for $f \in C_0^{\infty}(\R^n \setminus 0) $. Then,   (\ref{proof_lemma_resolvent_6_6}) allows to use the Gronwall argument. We conclude  using the density of $ C_0^{\infty} (\R^n \setminus 0) $ in  $ {\mathcal G}^1 $ and the fact that  $ ||S^{1/2} f ||_{L^2} \le C || f ||_{{\mathcal G}^1} $.
\end{proof}

With this proposition at hand, we can define $ B := S^{1/2} (A+i\kappa)^{-1} $ as an operator from $ {\mathcal G}^1 $ to $ L^2 $ and then define the bounded operator $ F_{\epsilon} (z) : L^2 \rightarrow L^2 $ by
$ F_{\epsilon} (z) := B G_{\epsilon} (z) B^* . $ 
 It is useful to record that  (\ref{definitionpreciseadjoint}) implies that
\begin{align}
 F_{\epsilon}(z)^* = F_{-\epsilon} (\bar{z}) \label{adjointrecord}
\end{align} 

 Seeing $ S^{1/2} $ as an operator from $ {\mathcal G}^1 $ to $ L^2 $, we denote its adjoint as $ (S^{1/2})^* $ (it maps conjugate linear forms on $ L^2 $ to conjugate linear forms on $ {\mathcal G}^1 $).
Notice that, on $ D (S^{1/2}) \subset L^2 $, $ (S^{1/2})^* $ coincides with $ S^{1/2} $  in the sense that $ (S^{1/2})^* \langle f , . \rangle = \langle S^{1/2} f , . \rangle $ as elements of $ {\mathcal G}^{-1} $. We use the notation $ (S^{1/2})^* $  to distinguish clearly  $ S^{1/2} : {\mathcal G}^1 (\subset D (S^{1/2})) \rightarrow L^2 $ from $ ( S^{1/2})^*: L^2 \rightarrow {\mathcal G}^{-1}  $.

\begin{proposition} \label{algebretechnique} Consider $ \tilde{S}^{\prime} :  {\mathcal G}^1 \rightarrow {\mathcal G}^{-1} $  introduced in (\ref{introductionSprime}). Then  for $ \epsilon {\emph Im}(z) > 0 $,
$$ \frac{d}{d\epsilon} F_{\epsilon} (z) = 2 i \kappa F_{\epsilon} (z) - S^{1/2} G_{\epsilon} (z) B^* + B G_{\epsilon}(z) (S^{1/2})^*  - \epsilon B G_{z}(\epsilon) \tilde{S}^{\prime} G_{z} (\epsilon) B^* .$$
\end{proposition}

\begin{proof}
We note first that $ e^{itA} $ is strongly continuous on $ {\mathcal G}^1 $. This is the case on $ C_0^{\infty} (\R^n \setminus 0) $ (for the $ {\mathcal G}^1 $ topology) according to the proof of Proposition \ref{lemma_resolvent_6} and remains true on $ {\mathcal G}^1 $ by density and the locally uniform bound $ ||e^{itA}||_{{\mathcal G}^1 \rightarrow {\mathcal G}^1} \leq e^{c_0 |t|} $. Using (\ref{expressionaderiverA}), whose integrand is continuous by the strong continuity of $ e^{itA} $ on $ {\mathcal G}^1 $, we find that
\begin{align}
 Q_{[H,iA]} (f,g) = \left. \frac{d}{dt} Q_H (e^{itA}f , e^{itA}g) \right|_{t=0} \label{above1}
\end{align}
 for all $ f, g \in {\mathcal G}^1 $. Note that we do not use (nor claim) that $ e^{itA}f $ and $ e^{itA}g $ are differentiable at $t=0$ for any $f,g \in {\mathcal G}^1$. Similarly, for all $ f, g \in {\mathcal G}^1 $,
\begin{align}
    Q_{[[H,iA],iA]} (f,g) = \left. \frac{d}{dt} Q_{[H,iA]} (e^{itA}f , e^{itA}g) \right|_{t=0} .  \label{above2}
\end{align} 
   Define $ e^{-itA} $ on $ {\mathcal G}^{-1} $ by $ \langle e^{-itA} u , f \rangle_{{\mathcal G}^{-1},{\mathcal G}^1} =  \langle  u , e^{itA} f \rangle_{{\mathcal G}^{-1},{\mathcal G}^1}  $ so that it is a bounded operator on $ {\mathcal G}^{-1} $, with bounded inverse $ e^{itA} $. This allows to define the $t$ dependent families of operators
 $$  L_t := \tilde{H}_t  - i \epsilon \tilde{S}_t, \quad  \tilde{H}_t := e^{-itA}  \tilde{H}  e^{itA} , \quad \tilde{S}_t = e^{-itA} \tilde{S} e^{itA} . $$
Then (\ref{above1}) and (\ref{above2}) show that these families are weakly differentiable at $t = 0$. In particular
\begin{align}
 \left. \frac{d}{dt}  L_t \right|_{t=0} =  \tilde{S} - i \epsilon \tilde{S}^{\prime} .  \label{deriveeL}
\end{align}
By Proposition \ref{propinversibilite},  the operator $ L_t - z $ is invertible with inverse $ G^t_{\epsilon} (z) := e^{-itA} G_{\epsilon} (z) e^{itA} $.
  By uniform boundedness principle, the weak differentiability of $ L_t $ at $t=0$ implies  that
$$ \big| \big| L_t - L_0 \big| \big|_{{\mathcal G}^1 \rightarrow {\mathcal G}^{-1}} = O (t) , \quad  \big| \big| G_{\epsilon}^t(z) - G_{\epsilon} (z) \big| \big|_{{\mathcal G}^{-1} \rightarrow {\mathcal G}^{1}} = O (t) $$
the second estimate being a consequence of the first one.  Then
\begin{align}
  \frac{1}{t} \big( G^t_{\epsilon} (z)  - G_{\epsilon} (z) \big) =  G_z (\epsilon) \frac{1}{t} \big( L_0 - L_t \big) G_{\epsilon}(z) +    \big( G_z^t (\epsilon) - G_z (\epsilon) \big) \frac{1}{t} \big(  L_0 - L_t  \big) G_{\epsilon}(z)  \label{BanachSteinhaus}
 \end{align}
 where the second term in the right hand side is $ O (t) $ in the $ {\mathcal G}^{-1} \rightarrow {\mathcal G}^{1} $ operator norm by (\ref{BanachSteinhaus}). On the other hand, using (\ref{definitionpreciseadjoint}) and (\ref{deriveeL}), it is easy to see that for any $ u , v \in {\mathcal G}^{-1} $,
$$ \big\langle u ,  G_z (\epsilon) \frac{1}{t} \big( L_0 - L_t \big) G_{\epsilon}(z) v \big\rangle_{{\mathcal G}^{-1}, {\mathcal G}^1} \rightarrow - \big\langle u ,  G_z (\epsilon)  \big( \tilde{S} - i \epsilon \tilde{S}^{\prime} \big) G_{\epsilon}(z) v \big\rangle_{{\mathcal G}^{-1}, {\mathcal G}^1}   \quad\text{as}\quad t \rightarrow 0 . $$
 In other words,  $ G_{\epsilon}^t (z) $ is weakly differentiable at $t=0$  with derivative $ G_z (\epsilon) ( i \epsilon \tilde{S}^{\prime} - \tilde{S} ) G_z (\epsilon) $. Taking into account the formula (\ref{formulederiveeG}), we find that
 $$ \frac{d}{d \epsilon} B G_{\epsilon}(z) B^* = -i \left. \frac{d}{dt} B e^{-itA} G_{\epsilon} (z) e^{itA} B^* \right|_{t=0} -  \epsilon B G_{\epsilon}(z)  \tilde{S}^{\prime} G_{\epsilon}(z) B^* . $$
 This formula is true for any bounded operator $ B : {\mathcal G}^1 \rightarrow L^2 $. For $ B = S^{1/2} (A + i \kappa)^{-1} $, one can test the above identity against $ f , g \in {\mathcal G}^{1} $ (which is dense in $ L^2 $) so that, by using
 $$ \left. \frac{d}{dt} e^{itA} B^* g \right|_{t=0} = i A (A-i\kappa)^{-1} S^{1/2} g = S^{1/2}g - i \kappa B^* g , $$
 we obtain easily the result.
 \end{proof}
 
\begin{proof}[Proof of Theorem \ref{theorem_resolvent_7}]
For simplicity, we denote by $ || \cdot || $ the operator norm on $ L^2 $.
Since $ || F_{\epsilon} (z) || \leq C_{\kappa} |\epsilon|^{-1/2}  || F_{\epsilon} (z) ||^{1/2} $ by Proposition \ref{bornedissipative} and (\ref{pourinitialisation}), we have the following estimate uniform in $z$ such that $ \epsilon \mbox{Im} (z) > 0$,
\begin{align}
  || F_{\epsilon} (z) ||\leq C_{\kappa}^2 |\epsilon|^{-1} . \label{bornedinitialisation} 
\end{align}  
On the other hand, using Proposition \ref{algebretechnique}, the norm $ \big| \big| \frac{d}{d \epsilon} F_{\epsilon} (z) \big| \big| $   is bounded (from above) by
$$   2 \kappa \big| \big|F_{\epsilon}(z) \big| \big| + \big| \big| S^{1/2} G_{\epsilon} (z) B^* \big| \big| + \big| \big| S^{1/2} G_{-\epsilon} (\bar{z}) B^* \big| \big| + C^{\prime} |\epsilon|   \big| \big| S^{1/2} G_{-\epsilon} (\bar{z}) B^* \big| \big|  \big| \big| S^{1/2} G_{\epsilon} (z) B^* \big| \big|  $$
with $C^{\prime}$ such that $ | Q_{[[H,iA],iA]}(f,g) | \leq C^{\prime} || S^{1/2} f ||_{L^2} || S^{1/2} g ||_{L^2} $. This is obtained easily by testing the expression of Proposition \ref{algebretechnique} and by using (\ref{definitionpreciseadjoint}). From Proposition \ref{bornedissipative} and (\ref{adjointrecord}), we obtain
$$ \left|  \left| \frac{d}{d \epsilon} F_{\epsilon} (z) \right| \right| \leq \big( 2 \kappa + C^{\prime} \big) \big| \big|F_{\epsilon}(z) \big| \big| + 2 |\epsilon|^{-1/2} \big| \big| F_{\epsilon} (z) \big| \big|^{1/2} .  $$
Together with (\ref{bornedinitialisation}), this gives a uniform bound $ || F_{\epsilon} (z) || \leq C $ for $ \epsilon \mbox{Im}(z) > 0 $. Since $ G_{\epsilon} (z) $ and $ F_{\epsilon} (z) $ are continuous up to $ \epsilon =0 $ (as $ \tilde{H} - z - i \epsilon \tilde{S} $ is), one obtains the result by letting $ \epsilon\to0$ in
$$ \left| \langle B^* f , G_{\epsilon} (z) B^* g \rangle_{{\mathcal G}^{-1},{\mathcal G}^1} \right| =  |\langle f, F_{\epsilon}(z) g \rangle |  \leq C || f ||_{L^2} ||g||_{L^2} $$
and by using that, for $ f \in {\mathcal G}^1 $, $ B^* f $ is given by the $L^2$ function $ (A- i \kappa)^{-1} S^{1/2} f $. 
\end{proof}

\appendix

\section{The Christ-Kiselev lemma}\label{Christ_Kiselev}\setcounter{equation}{0}
We record a special case  taken from  \cite[Lemma 3.1]{SmSo} of the Christ-Kiselev lemma \cite{CK}. \begin{lemma}\label{lemma_CK} Let $a,b\in \R$ and let $X$ and $Y$ be Banach spaces. Consider the integral operator $$Tf(t)=\int_a^b K(t,s)f(s)ds.$$ Suppose that $K\in L^1_{\mathrm{loc}}(\R^2,\mathbb B(X,Y))$ and $T$ is bounded from $L^p([a,b];X)$ to $L^q([a,b];Y)$ and satisfies $$|| Tf ||_{L^{q}([a,b];Y)}\le C_0 || f ||_{L^{p}([a,b];X)}$$
 for some $1\le p<q\le\infty$ and $C_0>0$. Then the operator $\widetilde T$ defined by $$\widetilde{T}f(t)=\int_a^tK(t,s)f(s)ds$$ is also bounded from $L^p([a,b];X)$ to $L^q(a,b];Y)$ and satisfies $$\big| \big| \widetilde{T} f \big| \big|_{L^{q}([a,b];Y)}\le C_1|| f ||_{L^{p}([a,b];X)},$$where $C_1=C_02^{1-2(1/p-1/q)}(1-2^{-(1/p-1/q)})^{-1}$. \end{lemma}Note that the condition $p<q$ is essential in the sense that if $K(t,s)=(t-s)^{-1}$ then this lemma fails for $1<p=q<\infty$.

\section{Proof of Theorem \ref{theorem_resolvent_2main}}
\label{appendix_C}
\setcounter{equation}{0}

Here we prove Theorem \ref{theorem_resolvent_2main}. It will be convenient to use the notation $ r = |x| $ and $ \partial_r = \frac{x}{|x|} \cdot \nabla $.

Let $f\in C_0^{\infty}(\R^n \setminus 0)$, $z=\lambda+i \varepsilon \in \C\setminus[0,\infty)$ with $\lambda,\varepsilon\in\R$. Let $ u = (H - \lambda - i \varepsilon)^{-1} f $ be the solution to the Helmholtz equation
\begin{align}
\label{Helmholtz}
u = (H-\lambda- i \varepsilon) f .
\end{align} 
Note that $H$ is nonnegative (by the assumptions \eqref{assumption_1_1} or \eqref{positifdim2}) so we may take $ \varepsilon = 0 $ if $ \lambda < 0 $.
Below, we only consider the case $ \varepsilon \geq 0$ (i.e. $ \varepsilon > 0 $ if $ \lambda \geq 0 $ or $ \varepsilon \geq 0 $ if $ \lambda < 0 $) since the proof for the case $ \varepsilon <0$ is analogous. The proof basically follows the method of \cite[Sections 2 and 3]{BVZ} which is based on the following two lemmas.

\begin{lemma} Let $ n \geq 2 $.
\label{lemma_appendix_C_1}
Then $ r^{\frac{1}{2}}u $ and $ r^{\frac{1}{2}} \nabla u $ belong to $ L^2 $ and we have the following five identities
\begin{align}
\label{proof_C_2}
\int\Big(|\nabla u|^2-\lambda |u|^2+V|u|^2\Big)dx
&=\emph{Re}\int f\overline{u}dx,\\
\label{proof_C_3}
- \varepsilon \int|u|^2dx
&=\emph{Im} \int f\overline{u}dx,\\
\label{proof_C_4}
\int \Big(r|\nabla u|^2-\lambda r|u|^2+ rV|u|^2+ \emph{Re}(\overline u\partial_ru)\Big)dx
&=\emph{Re} \int rf\overline udx,\\
\label{proof_C_5}
\int\Big(- \varepsilon r|u|^2+ \emph{Im}(\overline u\partial_ru)\Big)dx
&= \emph{Im} \int rf\overline udx,\\
\label{proof_C_6}
\int\Big(2|\nabla u|^2- (r\partial_r V)|u|^2-2 \varepsilon \emph{Im}(\overline ur\partial_ru)\Big)dx
&= \emph{Re} \int f(2r\partial_r \overline u+n\overline u)dx. 
\end{align}
\end{lemma}

\begin{proof}
Note that $\mathcal G^1=\mathcal H^1$ if $n\ge3$ and $\mathcal G^1$ is the completion of $C_0^\infty(\R^2\setminus\{0\})$ with respect to the norm $(Q_H(u)+||u||_{L^2}^2)^{1/2}$ if $n=2$ under conditions in Theorem \ref{theorem_resolvent_2main}. 

\eqref{proof_C_2} and \eqref{proof_C_3} just correspond to the expressions of  the real and imaginary parts of the identity $ Q_H (u,u) - z ||u||_{L^2}  = \langle f , u \rangle$ which follows from (\ref{Helmholtz}). We point out that the integral $ \int V |u|^2 dx $ is well defined thanks to \eqref{theorem_resolvent_2main_1} if $n\ge3$. 
If $n=2$, we use that $ V^{1/2} u  $ belongs to $ L^2 $ for $u \in {\mathcal G}^1 $ since if $ u_j \in C_0^{\infty} (\R^2 \setminus 0) $ approaches $ u $  in $ {\mathcal G}^1 $, then $ V^{1/2} u_j $ is a Cauchy sequence in $ L^2 $. 

At a formal level, (\ref{proof_C_3}) and (\ref{proof_C_4}) follow by multiplying (\ref{Helmholtz}) by $ r \bar{u}  $, then integrating and taking real and imaginary parts. To make this calculation rigorous, we pick $ \chi \in C_0^{\infty}(\R) $ equal to $1$ near $0$ and multiply   (\ref{Helmholtz}) by $ r \chi (\delta
  r) \bar{u} =: r \chi_{\delta} \bar{u}  $. It is not hard to check that $ r \chi_{\delta} u \in {\mathcal G}^1 $ which allows to use the identity $ Q_H (u, r \chi_{\delta} u) = \langle{H u ,r \chi_{\delta} u \rangle}  $. Taking  the imaginary part (and using (\ref{Helmholtz})), we obtain
$$ \int - \varepsilon r \chi_{\delta} |u|^2 dx
 = \mbox{Im} \int  \Big(  \chi_{\delta} rf \bar{u} -  ( r \chi_{\delta} )^{\prime} \bar{u}\partial_ru \Big) dx . $$
 Since the right hand side has a limit as $ \delta \rightarrow 0 $ while the integrand of the left hand side has a fixed sign, we can let $ \delta \rightarrow 0 $ and get (\ref{proof_C_5}) by monotone convergence (we can choose $ \chi $ such that $ \chi_{\delta} (r) \uparrow 1 $ as $ \delta \downarrow 0 $). In particular, we have $r^{1/2}u\in L^2$. Taking next the real part of  $ Q_H (u, r \chi_{\delta} u) = \langle{H u ,r \chi_{\delta} u \rangle}  $, we have
 $$ \int  \chi_{\delta }r|\nabla u|^2 dx =  \int \lambda r \chi_{\delta} |u|^2 - r \chi_{\delta} V|u|^2 - \mbox{Re} \big( (r \chi_{\delta})^{\prime}  \bar{u}\partial_ru \big) dx + \mbox{Re} \int \chi_{\delta} rf \bar{u}dx, $$
 whose right hand side converges as $ \delta \rightarrow 0 $ since we have already shown that $ r^{1/2} u \in L^2 $ while $ r V |u|^2 $ is integrable by the Cauchy-Schwarz inequality and \eqref{theorem_resolvent_2main_1} if $n\ge3$ or Assumption \ref{assumption_2} if $n=2$. 
 Letting $ \delta \rightarrow 0 $, we get (\ref{proof_C_4}). In particular, it shows that $ r^{\frac{1}{2}} \nabla u \in L^2 $.
 
 It remains to prove (\ref{proof_C_6}). Formally, it is  obtained by multiplying (\ref{Helmholtz}) by $ i A \bar{u} $, integrating and taking the real part. However $ A \bar{u} $ does not clearly belong to $ L^2 $ (we do not know that $r \partial_r u \in L^2 $), so we need more arguments to justify the formula. For $ \delta > 0 $ we replace $  A u $ by $ A (\delta A^2 + 1)^{-1} u $.  Note that, by Proposition \ref{lemma_resolvent_6}, $ {\mathcal G}^1 $ is stable by $ A(\delta A^2 + 1)^{-1} = A (\delta^{1/2}A+i)^{-1} (\delta^{1/2} A - i)^{-1} $. Using (\ref{Helmholtz}) and the fact 
 that $ f \in C_0^{\infty} (\R
 ^n) \subset D (A) $, we have   first
 \begin{align}
  2 \mbox{Re} \langle H u , i A (\delta A^2  + 1)^{-1} u \rangle = 2 \varepsilon \langle u , A (\delta A^2 + 1)^{-1} u \rangle + \mbox{Im} \langle A f , (\delta A^2 + 1) u \rangle .  \nonumber
 \end{align}
 Using that $ r^{1/2} u , r^{1/2} \partial_r u \in L^2 $ and that $ (\delta^{1/2}A\pm i)^{-1} r^{\frac{1}{2}} =  r^{\frac{1}{2}} \big(\delta^{1/2}A \pm i(1 \mp \delta^{1/2}/2) \big)^{-1}  $, we can let $ \delta \rightarrow 0 $ in this identity so that
 \begin{align}
    2 \mbox{Re} \langle H u , i A (\delta A^2  + 1)^{-1} u \rangle \rightarrow   2  \varepsilon \mbox{Im} \int  r^{\frac{1}{2}} \bar{u} r^{\frac{1}{2}} \partial_r u  dx  + \mbox{Im} \int  Af \bar{u} dx . \label{premierepartiereelle}
 \end{align}   
On the other hand,  
 since  $ A (\delta A^2 + 1)^{-1} u $ belongs to $ {\mathcal G}^1  $
 one can write
\begin{align}
  2 \mbox{Re} \langle H u , i A (\delta A^2  + 1)^{-1} u \rangle =  iQ_{H} \big(A (\delta A^2 + 1)^{-1} u , u \big)  - i Q_H (u,A (\delta A^2 + 1)^{-1} u ) .  \label{deuxiemepartiereelle}
\end{align}  
To let $ \delta \rightarrow 0 $ in this expression, we study separately the contribution of $ - \Delta $ and of $V$.
It is not hard to check that $ (\delta A^2 + 1)^{-1} u \rightarrow u $ in $ {\mathcal G}^1 $ as $ \delta \rightarrow 0 $ (by writing the resolvent of $A$ in term of $ e^{itA} $ as in the proof of Proposition \ref{lemma_resolvent_6}). Let $ Q_H = Q_{H_0} $ be the quadratic form associated to the Laplacian, i.e. to $V= 0$. Setting $ u_{\delta} = (\delta^{\frac{1}{2}}A + i)^{-1} u $ and using the  formulas  $[ \partial_j , (\delta^{\frac{1}{2}}A \pm i )^{-1} ] =i \delta^{\frac{1}{2}}(\delta^{\frac{1}{2}}A \pm i )^{-1} \partial_j (\delta^{\frac{1}{2}} A \pm i)^{-1} $, it is not hard to check that
\begin{align*}
 i Q_{H_0} \big(A (\delta A^2 + 1)^{-1} u , u \big)  =  i Q_{H_0} \big( A u_{\delta} , u_{\delta} \big) + O \Big( \big| \big|\delta^{\frac{1}{2}} A u_{\delta} \big| \big|_{{\mathcal H}^1} || u ||_{{\mathcal H}^1} \Big) .
\end{align*}
We omit the details such as the possible approximation of $u$ by a $ C_0^{\infty} $ function in $ {\mathcal G}^1 $. Next, we observe that $ \delta^{\frac{1}{2}} A u_{\delta} \rightarrow 0 $ in $ {\mathcal H}^1 $ as $ \delta \rightarrow 0 $: this is obvious on $ L^2 $ by the spectral theorem and remains true on $ {\mathcal H}^1 $ by using that $ \partial_j  (\delta^{\frac{1}{2}}A+i)^{-1} =   (\delta^{\frac{1}{2}}A+i- \delta^{\frac{1}{2}} i)^{-1} \partial_j  $. Thus, in the right hand side of (\ref{deuxiemepartiereelle}), the contribution of $ Q_{H_0} $ as $ \delta \rightarrow 0 $ is $ Q_{[H_0,iA]} (u,u) = 2 ||\nabla u||^2_{L^2}$. We next study the contribution of $V$ in the right hand side of  (\ref{deuxiemepartiereelle}) which reads
\begin{align}
 2 \int V \mbox{Re} \Big(  \bar{u}  (r \partial_r - n/2) (\delta A^2 + 1)^{-1} u \Big)  dx \stackrel{\delta \downarrow 0}{ \longrightarrow } 2 \int V \mbox{Re} \Big(  \bar{u}  (r \partial_r - n/2)  u \Big)  dx  =  \int V \nabla \cdot x ( |u|^2 )  dx .
 \nonumber
\end{align}
To take the limit $ \delta \downarrow 0 $, we use, when $ n \geq 3 $, that $ || r V \bar{u} ||_{L^2} \le C || u||_{\mathcal H^1} $ and that $ (\delta A^2 + 1)^{-1} $ goes strongly to $ 1 $ in $ {\mathcal H}^1 $. When $n=2$, we use that $ || r V \bar{u} ||_{L^2} \leq || r V^{\frac{1}{2}} ||_{L^{\infty}} ||V^{\frac{1}{2}} u||_{L^2}$ with $ ||V^{\frac{1}{2}} u||_{L^2} < \infty  $ for $ u \in {\mathcal G}^1 $ and that $ (\delta A^2 + 1 )^{-1} \rightarrow 1 $ strongly in $ {\mathcal G}^1 $ (to control the term with $n/2$). This is obtained by using $ \int_0^{\infty} e^{-t} ||e^{it\delta^{\frac{1}{2}}A} u - u ||_{{\mathcal G}^1} dt \rightarrow 0$, by dominated convergence (see Proposition \ref{lemma_resolvent_6}) and the strong continuity of $ e^{itA} $ on $ {\mathcal G}^1 $.
 To integrate by part in the limit and rewrite it as $ - \int \big( r \partial_r  V \big) |u|^2 dx $, we use that $  \int V \nabla \cdot x ( |v|^2 )  dx  $ depends continuously on $v \in {\mathcal G}^1 $ for the same reasons as the above convergence, and then approximate  $ v $ by $ C_0^{\infty} $ functions so that the integration by part holds in the sense of distributions and remains true in the limit since $  v \mapsto - \int \big( r \partial_r  V \big) |v|^2 dx  $ is also continuous on $ {\mathcal G}^1 $ thanks to the assumption that $ || r^{1/2} |\partial_r V|^{1/2} f ||_{L^2} \le C || f ||_{\mathcal H^1} $.
To sum up, we have  shown that the right hand side of (\ref{deuxiemepartiereelle}) goes to $ 2 ||\nabla u ||_{L^2}^2 - \int (r \partial_r V)|u|^2  dx$ as $ \delta \rightarrow 0 $ so, taking (\ref{premierepartiereelle}) into account, we obtain (\ref{proof_C_6}).
\end{proof}

 \begin{lemma}
\label{lemma_appendix_C_2}
Let $ n\geq 2  $,  $0<\varepsilon<\lambda$ and $v_\lambda=e^{-i\lambda^{\frac12}r}u$. Then one has
\begin{align}
\nonumber\int \Big(|\nabla v_\lambda|^2+ \varepsilon \lambda^{-\frac12}r|\nabla v_\lambda|^2\Big)dx
&=\int\Big(\partial_r(rV)|v_\lambda|^2
-\varepsilon\lambda^{-\frac12}rV|v_\lambda|^2
-\varepsilon\lambda^{-\frac12}\emph{Re}(\overline ue^{i\lambda r}\partial_rv_\lambda)\Big)dx\\
\label{proof_C_7}
&\quad\ +\emph{Re}\int\Big((n-1)f\overline u+\varepsilon \lambda^{-\frac12}rf\overline u+2rf\overline{e^{i\lambda^{\frac12}r}\partial_rv_\lambda}\Big)dx. 
\end{align}
\end{lemma}

\begin{proof}
At first observe from the identity $|z-iw|^2=|z|^2+|w|^2-2 \mbox{Im}(z\overline{w})$ for $z,w\in\C$ that
\begin{align}
\label{proof_C_8}
|\nabla v_\lambda|^2=|\nabla u-i\lambda^{\frac12}|x|^{-1}x u|^2=|\nabla u|^2+\lambda |u|^2-2\lambda^{\frac12} \mbox{Im}[(\partial_r u)\overline u]. 
\end{align}
Then the formula \eqref{proof_C_7} is derived by computing
$\eqref{proof_C_6}-\eqref{proof_C_2}-2\lambda^{\frac12}\times\eqref{proof_C_5}+\varepsilon \lambda^{-\frac12}\times\eqref{proof_C_4}$ as follows. First, taking \eqref{proof_C_8} into account, $\eqref{proof_C_6}-\eqref{proof_C_2}$ reads
\begin{equation}
\begin{aligned}
\label{proof_C_8_1}
\int\Big(|\nabla v_\lambda|^2+2\lambda^{\frac12} \mbox{Im}[(\partial_r u)\overline u]-(\partial_r (rV))|u|^2-2 \varepsilon \mbox{Im} [(r\partial_ru)\overline{u}]\Big)dx\\
=\mbox{Re}\int\Big(f(2r\partial_r\overline u+(n-1)\overline u)\Big)dx. 
\end{aligned}
\end{equation}
To eliminate $2\lambda^{\frac12} \mbox{Im}[(\partial_r u)\overline u]$, we subtract $2\lambda^{\frac12}\times\eqref{proof_C_5}$ from \eqref{proof_C_8_1} to obtain
\begin{equation}
\begin{aligned}
\label{proof_C_9}
\int\Big(|\nabla v_\lambda|^2-(\partial_r (rV))|u|^2-2 \varepsilon \mbox{Im} [(r\partial_ru)\overline{u}]+2 \varepsilon \lambda^{\frac12} r|u|^2\Big)dx\\
=\int\Big(\mbox{Re}[f(2r\partial_r\overline u+(n-1)\overline u)]-2\lambda^{\frac12} \mbox{Im}  (rf\overline u)\Big)dx.
\end{aligned}
\end{equation}
Since $- \mbox{Im}  (rf\overline u)= \mbox{Re}(rf\overline{-iu})$ and $\partial_r u-i\lambda^{\frac12} u=e^{i\lambda^{\frac12} r}\partial_r v_\lambda$, the right hand side of \eqref{proof_C_9} reads
\begin{align}
\label{proof_C_9_1}
\mbox{Re}[f(2r\partial_r\overline u+(n-1)\overline u)]-2\lambda^{\frac12} \mbox{Im}  (rf\overline u)= \mbox{Re}\Big(2rf\overline{e^{i\lambda^{\frac12} r}\partial_r v_\lambda}+(n-1)f\overline u\Big). 
\end{align}
Using \eqref{proof_C_8} we next compute
\begin{equation}
\begin{aligned}
\label{proof_C_9_2}
-2\varepsilon \mbox{Im} [(r\partial_ru)\overline{u}]+2 \varepsilon \lambda^{\frac12} r|u|^2+\varepsilon \lambda^{-\frac12}(r|\nabla u|^2-\lambda r|u|^2)
= \varepsilon \lambda^{-\frac12}r|\nabla v_\lambda|^2.
\end{aligned}
\end{equation}
It is then seen from \eqref{proof_C_9_1} and \eqref{proof_C_9_2} that $\eqref{proof_C_9} + \varepsilon \lambda^{-\frac12}\times\eqref{proof_C_4}$ reads
\begin{equation}
\begin{aligned}
\label{proof_C_10}
\int\Big(|\nabla v_\lambda|^2-(\partial_r (rV))|u|^2+ \varepsilon \lambda^{-\frac12}r|\nabla v_\lambda|^2+\varepsilon \lambda^{-\frac12}rV|u|^2+\varepsilon \lambda^{-\frac12} \mbox{Re}(\overline u\partial_ru)\Big)dx\\
=\mbox{Re}\int\Big(2rf\overline{e^{i\lambda^{\frac12} r}\partial_r v_\lambda}+(n-1)f\overline u+\varepsilon\lambda^{-\frac12}rf\overline u\Big). 
\end{aligned}
\end{equation}
Finally, since $ \mbox{Re}(\overline u\partial_ru)=\mbox{Re}[\overline u(\partial_r u-i\lambda^{\frac12} u)]=\mbox{Re}(\overline ue^{i\lambda r}v_\lambda)$, \eqref{proof_C_10} is equivalent to \eqref{proof_C_7}. 
\end{proof}

\begin{proof}[Proof of Theorem \ref{theorem_resolvent_2main} (1)] Here we consider the case when  $ n \geq 3 $.
It suffices to show
\begin{align}
\label{proof_C_10_0}
||r^{-1}u||_{L^2}\le C||rf||_{L^2}
\end{align}
uniformly in $\lambda > 0$ and $\varepsilon>0$ or in $\lambda<0$ and $\varepsilon=0$. When $\varepsilon\ge\lambda>0$, \eqref{proof_C_2} and \eqref{proof_C_3} imply
\begin{align}
\label{proof_C_10_1}
\int\Big(|\nabla u|^2+V|u|^2\Big)dx\le (1+\lambda_+/\varepsilon)\int |fu|dx\le \delta_1 ||r^{-1}u||_{L^2}^2+\delta_1^{-1} ||rf||_{L^2}^2
\end{align}
for any $\delta_1>0$, where $\lambda_+=\max\{0,\lambda\}$. Note that if $\lambda< 0$ and $\varepsilon=0$, \eqref{proof_C_10_1} still holds with $1+\lambda_+/\varepsilon$ replaced by $1$. On the other hand, the hypothesis \eqref{assumption_1_1} and Hardy's inequality show
$$
\int\Big(|\nabla u|^2+V|u|^2\Big)dx\ge \delta_0\int |\nabla u|^2dx \ge \delta_0 C_{C_{\rm H}}\int r^{-2}|u|^2dx.
$$
Choosing $\delta_1>0$ so small that $\delta:=\delta_0 C_{C_{\rm H}}-\delta_1>0$ we obtain \eqref{proof_C_10_0}.  

We next let $\varepsilon<\lambda$. By Hardy's inequality, $||\nabla v_\lambda||_{L^2}\ge C_{C_{\rm H}}||r^{-1}v_\lambda||_{L^2}^2=C_{C_{\rm H}}||r^{-1}u||_{L^2}^2$. Hence it suffices to show \eqref{proof_C_10_0} that there exist $\delta,C_\delta>0$, independent of $\lambda$ and $\varepsilon$, such that the right hand side of \eqref{proof_C_7} is bounded from above by 
$(1-\delta)||\nabla v_\lambda||_{L^2}^2+(1-\delta)\varepsilon\lambda^{-\frac12} ||r^{\frac12}\nabla v_\lambda||_{L^2}^2+C_\delta||rf||_{L^2}^2$.  
By the hypothesis \eqref{assumption_1_2}, the first term of the right hand side of  \eqref{proof_C_7} satisfies
\begin{align}
\label{proof_C_12}
\int \partial_r(rV)|v_\lambda|^2dx
\le (1-\delta_0)||\nabla v_\lambda||_{L^2}^2 .
\end{align}
This allows to absorb the first term of the right hand side of  \eqref{proof_C_7} in the left hand side (of \eqref{proof_C_7}).
For the second term of the right hand side of  \eqref{proof_C_7}, it follows from \eqref{assumption_1_1} that 
\begin{align}
-\varepsilon\lambda^{-\frac12}\int rV|v_\lambda|^2dx\le \varepsilon\lambda^{-\frac12}(1-\delta_0) || \nabla (r^{\frac{1}{2}} v_{\lambda} ) ||_{L^2}^2
\label{avecHardyapoids}
\end{align}
provided we know that $ r^{\frac{1}{2}} v_{\lambda} $ belongs to $ {\mathcal H}^1 $. This follows from the fact that $ r^{\frac{1}{2}} u,r^{\frac{1}{2}}\nabla u\in L^2 $ by Lemma \ref{lemma_appendix_C_1} together with following weighted Hardy's inequality (see, {\it e.g.}, \cite[Proposition 8.1]{MSS} in which a simple proof can be found)
$$ || r^{-\frac{1}{2}} v ||_{L^2} \le C || r^{\frac{1}{2}} \nabla v ||_{L^2} . $$
Since $  \nabla (r^{\frac{1}{2}} v_{\lambda} ) =  r^{\frac{1}{2}} \nabla v_{\lambda} + \frac{1}{2} r^{-\frac{1}{2}} \frac{x}{|x|} v_{\lambda} $ and $\varepsilon<\lambda$, the right hand side of (\ref{avecHardyapoids}) is bounded by
 $$ \varepsilon\lambda^{-\frac12}(1-\delta_0) ||r^{\frac{1}{2}}  v_{\lambda}  ||_{L^2}^2 + C \varepsilon^{\frac{1}{2}} \left( \int |v_{\lambda}| |\partial_r v_{\lambda}| + r^{-1} |v_{\lambda}|^2 dx\right) .  $$
 The interest of this bound is that its first term can be absorbed in the left hand side of (\ref{proof_C_7}). For the other terms, using
 the Cauchy-Schwarz and Hardy inequalities, we can bound them by
 $$ \delta_1 || \nabla v_{\lambda} ||^2_{L^2} + C_1 \delta_1^{-1} \varepsilon || v_{\lambda} ||_{L^2}^2 . $$
for any $\delta_1>0$ with $C_1$ being independent of $\delta_1$ and $\varepsilon$.   
Then \eqref{proof_C_3} and Hardy's inequality imply
\begin{align*}
\varepsilon||v_\lambda||_{L^2}^2
=\varepsilon||u||_{L^2}^2
\le \int |fu|dx
\le C_1^{-1}\delta_1^2||\nabla v_\lambda||_{L^2}^2+C \delta_1^{-2} ||rf||_{L^2}^2.
\end{align*}
Summing up, we have shown  that (\ref{proof_C_7}) implies
\begin{align}
\nonumber \delta_0 \int \Big(|\nabla v_\lambda|^2+ \varepsilon \lambda^{-\frac12}r|\nabla v_\lambda|^2\Big)dx
& \leq 2 \delta_1 || \nabla v_{\lambda} ||_{L^2}^2 + C \delta_1^{-3} || r f  ||_{L^2}^2 + \int\Big(
-\varepsilon\lambda^{-\frac12}\mbox{Re}(\overline ue^{i\lambda r}\partial_rv_\lambda)\Big)dx\\
\label{reecriturelocale}
&\quad\ + \mbox{Re}\int\Big((n-1)f\overline u+\varepsilon \lambda^{-\frac12}rf\overline u+2rf\overline{e^{i\lambda^{\frac12}r}\partial_rv_\lambda}\Big)dx. 
\end{align}
To bound the two integrals in the right hand side, similar computations yield 
\begin{align}
\nonumber
\varepsilon\lambda^{-\frac12}\Big| \mbox{Re} \int e^{i\lambda^{\frac12} r}(\partial_rv_\lambda) \overline udx\Big|
\le \sqrt \varepsilon ||\nabla v_\lambda||_{L^2} ||v_\lambda||_{L^2}
&\le \delta_1 ||\nabla v_\lambda||_{L^2}^2+C\delta_1^{-3} ||rf||_{L^2}^2,\\
\nonumber
\Big|(n-1) \mbox{Re}\int f\overline udx\Big|
+\Big|2 \mbox{Re} \int rf\overline{e^{i\lambda^{\frac12}r}\partial_rv_\lambda}dx\Big|
&\le \delta_1 ||\nabla v_\lambda|||_{L^2}^2+C\delta_1^{-1} ||rf||_{L^2}^2,\\
\nonumber
\varepsilon\lambda^{-\frac12}\Big|\int rf\overline udx\Big|
\le \sqrt\varepsilon ||rf||_{L^2} ||u||_{L^2}^2
&\le \delta_1 ||\nabla v_\lambda||_{L^2}^2+C\delta_1^{-3} ||rf||_{L^2}^2. 
\end{align}
Together with (\ref{reecriturelocale}), these estimates show that 
$$
\delta_0 ||\nabla v_{\lambda}||_{L^2}^2 \leq 5\delta_1 ||\nabla v_\lambda||_{L^2}^2
+C\delta_1^{-3} ||rf||_{L^2}^2
$$
hence by  choosing $\delta_1$ so that 
$
\delta_0-5\delta_1>0
$, we obtain \eqref{proof_C_10_0} by using the Hardy inequality. 
\end{proof}

\begin{remark}
\label{remark_C}
It was claimed in \cite{BVZ} that $r^{-1}$ is $H$-supersmooth  under \eqref{theorem_resolvent_2main_1}, \eqref{assumption_1_1} and \eqref{assumption_1_2}. However, their argument used a weighted Hardy type inequality 
$$
\int rV|f|^2dx\le (1-\delta_0)||r^{\frac12}\nabla f||_{L^2},\quad \delta_0>0,\ f\in{\mathcal H}^1
$$
with an explicit constant $1-\delta_0$ to deal with the term $-\varepsilon\lambda^{-\frac12}\int rV|v_\lambda|^2dx$, 
which seems to be not an obvious consequence of \eqref{theorem_resolvent_2main_1}, \eqref{assumption_1_1} and \eqref{assumption_1_2}. 
\end{remark}

\begin{proof}[Proof of Theorem \ref{theorem_resolvent_2main} (2)] Next we consider the case $n=2$. 
It suffices to show 
\begin{align}
\label{proof_C_17}
||V^\frac12 u||_{L^2}\le C||V^{-\frac12}f||_{L^2}
\end{align}
uniformly in $\lambda\in\R$ and $\varepsilon>0$ or in $\lambda<0$ and $\varepsilon=0$, where we note that $V^{-\frac12}f\in L^2$ for $f\in C_0^\infty$ since $V^{-\frac12}\in L^2_{\mathrm{loc}}$. When $\varepsilon\ge \lambda$, \eqref{proof_C_10_1} implies, for any $\delta>0$, 
$$
||V^\frac12 u||_{L^2}^2\le 2\int|fu|dx\le \delta||V^\frac12 u||_{L^2}^2+\delta^{-1}||V^{-\frac12}f||_{L^2}^2. 
$$
Taking $\delta<1$ we obtain \eqref{proof_C_17}. When $\varepsilon<\lambda$, Assumption \ref{assumption_2} (2) and \eqref{proof_C_7} imply
\begin{align*}
||\nabla v_\lambda||_{L^2}^2+c||V^\frac12 v_\lambda||_{L^2}^2
\le \sqrt\varepsilon\int (|\overline u\partial_rv_\lambda|+|rf\overline u|)dx
+C\int (|f\overline u|+r|f\overline{\partial_rv_\lambda}|)dx
\end{align*}
with some $c,C>0$. As in the case when $n\ge3$, for any $\delta>0$ there exists $C_\delta>0$ such that
$$
||\nabla v_\lambda||_{L^2}^2+c||V^\frac12 v_\lambda||_{L^2}^2
\le \delta(||\nabla v_\lambda||_{L^2}^2+||V^\frac12v_\lambda||^2_{L^2})+C_\delta(||V^{-\frac12}f||_{L^2}^2+||rf||_{L^2}^2). 
$$
Choosing $\delta>0$ so small that $\delta<\min(1,c)$ and using the fact $r^2V\in L^\infty$, we obtain \eqref{proof_C_17}. 
\end{proof}

\end{document}